\documentclass{article}
\usepackage[
    a4paper,
    left=2.8cm,
    right=2.5cm,
    top=2.5cm,
    bottom=2.5cm,
    headsep=0.8cm,
    footskip=1cm
]{geometry}

\newcounter{case}
\renewcommand{\thecase}{Case~\arabic{case}}

\newcommand{\mycase}[2]{%
  \refstepcounter{case}%
  \noindent\textbf{\thecase.\ #1}\label{#2}%
  \par\nobreak\vspace{4mm}\noindent\ignorespaces%
}

\usepackage{graphicx} 
\usepackage{amsmath,amssymb}
\usepackage{mathrsfs}
\usepackage{amsthm}
\usepackage{subcaption}
\usepackage{comment}
\usepackage{float}
\usepackage{xcolor}
\usepackage{hyperref}
\usepackage{cleveref}
\usepackage{natbib}

\newcommand{\intfour}{[4]}
\newcommand{\inttwo}{[2]}
\newcommand{\intthree}{[3]}

\newtheorem{theorem}{Theorem}[section]
\newtheorem{Definition}{Definition}[section]

\newtheorem{proposition}[theorem]{Proposition}
\newtheorem{problem}{Problem}[section] 
\newtheorem{claim}{Claim}[section]

\newtheorem{conjecture}{Conjecture}[section]

\newcommand{\jbj}[1]{{\color{red}#1}}
\newcommand{\jbjnote}[1]{{\color{magenta}JBJ: #1}}

\newcommand{\xyc}[1]{{\color{orange}#1}}

\title{On the $2$-Linkage Problem for Split Digraphs}

\author{
Xiaoying Chen\textsuperscript{1} \and
J{\o}rgen Bang-Jensen\textsuperscript{1,2} \and
Jin Yan\textsuperscript{1}  \and
Jia Zhou\textsuperscript{1}\thanks{Corresponding author. \\ Email address: jiazhou\_99@163.com}
}
\date{}

\begin{document}

\maketitle

\noindent\textsuperscript{1} School of Mathematics, Shandong University, Jinan 250100, China\\
\noindent\textsuperscript{2}Department of Mathematics and Computer Science, University of Southern Denmark, Odense DK-5230, Denmark

\vspace{1em}

\maketitle

\begin{abstract}

A digraph is {\bf \( k \)-linked} if for arbitary two disjoint vertex sets \(\{s_1, \ldots, s_k\}\) and \(\{t_1, \ldots, t_k\}\), there exist vertex-disjoint directed paths \(P_1, \ldots, P_k\) {such that \(P_i\) is a directed path from \(s_i\) to \(t_i\) for each $i\in [k]$}. A {\bf split digraph} is a digraph \( D = (V_1, V_2; A) \) whose vertex set is a disjoint union of two nonempty sets \( V_1 \) and \( V_2 \) such that \( V_1 \) is an independent set and the subdigraph induced by \( V_2 \) is semicomplete (no pair of non-adjacent vertices). A {\bf semicomplete split digraph} is a split digraph \( D = (V_1, V_2; A) \)  in which every vertex in the independent set \( V_1 \) is adjacent to every vertex in \( V_2 \). {Semicomplete split digraphs form an important subclass of the class of semicomplete multipartite digraphs.} In this paper, we prove that every 6-strong split digraph is 2-linked.  This solves a problem posed by Bang-Jensen and Wang [J. Graph Theory, 2025]. We also show that every 5-strong semicomplete split digraph is 2-linked. This bound is tight already for semicomplete digraphs.
\end{abstract}


\section{Introduction} 

Given a positive integer \( k \), a (di)graph is ({\bf strongly}) {\bf \(k\)-connected} (briefly {\bf \(k\)-strong} for digraphs) if it has at least \(k + 1\) vertices and, after deleting any set \(S\) of at most \(k - 1\) vertices, there is a path from \(x\) to \(y\) for every ordered pair of distinct vertices \(x\) and \(y\). A (di)graph \(D\) is {\bf\(k\)-linked} if for any
\(2k\) distinct vertices \(s_1, \dots, s_k, t_1, \dots, t_k\) in \(D\), there exist \(k\) disjoint paths \(P_1, \dots, P_k\) such that {\(P_i\) is a path from \(s_i\) to \(t_i\) for each \(i \in [k]\)}. It is easy to verify that every \(k\)-linked (di)graph is (\(k\)-strong) \(k\)-connected, but the converse does not hold.

{The study of $k$-linked (di)graphs has been a central topic in graph theory.} Larman and Mani \cite{larman1970existence} and Jung \cite{Jung1970} were the first to establish that there exists a function \( f(k) \) such that every \( f(k) \)-connected graph is \( k \)-linked. The initial bound on \( f(k) \) was exponential in \( k \); later, Bollobás and Thomason \cite{Bollobas1996} provided the first linear bound, showing that \( f(k) \leq 22k \). The best known bound is \( f(k) \leq 10k \), proved by Thomas and Wollan \cite{Thomas2005}. Notably, in 1980, {Seymour \cite{SEYMOUR1980293} and Thomassen \cite{thomassen19802-linked} independently characterized $2$-linked graphs by planarity, and thus proved that every $6$-connected graph is $2$-linked. Motivated by these results, Thomassen \cite{thomassen19802-linked} conjectured that an analogous statement holds for digraphs, and thus posed a stronger conjecture for digraphs.}

\begin{conjecture}[\cite{thomassen19802-linked}]\label{Conj1}
There exists a function \( f(k) \) such that every \( f(k) \)-strong digraph \( D \) is \( k \)-linked.
\end{conjecture}

{In 1991}, Thomassen \cite{thomassen1991highly} disproved the conjecture by constructing digraphs of arbitrarily high strong connectivity that are not even 2-linked. This illustrates a significant structural difference between graphs and digraphs. This difference is also displayed in the complexity of linkage problems. Fortune et. al \cite{fortuneTCS10} proved that already the 2-linkage problem is NP-complete for digraphs while one of the central results of the Robertson-Seymour Graph Minors project is that the $k$-linkage problem is polynomial for every fixed $k$ for graphs \cite{robertsonJCT63}.   This difference  has motivated the study of linkage problems for special classes of digraphs. A digraph is a {\bf tournament} if for every pair of distinct vertices \(u, v\), exactly one of \(uv, vu\) is an {arc}. {Moreover,} a digraph is {\bf semicomplete} if for all distinct {vertices} \(u, v\), at least one of \(uv, vu\) is an {arc}. Hence tournaments form a subclass of semicomplete digraphs.
Thomassen \cite{thomassen1983connectivity} proved that there exists an exponential function function $f(k)$ such that every $f(k)$-strong tournament is $k$-linked.  For 2-linkage problem, Bang-Jensen \cite{bang19882} established a tight bound for semicomplete digraphs, which is formally stated in the following theorem.

\begin{theorem}\cite{bang19882}\label{5-strong semicomplete digraph is 2-linked}
    Every 5-strong semicomplete digraph is \(2\)-linked. Furthermore, there exists an infinite class of \(4\)-strong tournaments which are not \(2\)-linked.
\end{theorem}

{In this paper}, we consider split digraphs and semicomplete split digraphs (the later name was first used  by Bang-Jensen and Wang in \cite{bang2025strong} and we point out that the definition of split digraphs used here differs from that in \cite{lamar2012split}, where the class of split digraphs was first introduced.). A {\bf split digraph} \( D = (V_1, V_2; A) \) is a digraph whose vertex set is a disjoint union of two nonempty sets \( V_1 \) and \( V_2 \) such that \( V_1 \) is an independent set and the subdigraph induced by \( V_2 \) is semicomplete, which is a natural generalization of the class of semicomplete digraphs. By definition, a split digraph is a digraph whose underlying undirected graph is a split graph \cite{stephane1977split}. A split digraph \( D = (V_1, V_2; A) \) is {\bf semicomplete} if every vertex in the independent set \( V_1 \) is adjacent to every vertex in \( V_2 \). 

Problems on split digraphs are often much harder than the corresponding problem for semicomplete digraphs: for instance, the hamiltonian cycle problem is NP-complete for general split digraphs, while it is polynomial for semicomplete split digraphs. The later follows from \cite{bangJGT29} and the fact that semicomplete split digraphs also semicomplete multipartite. As another example, the  2-linkage problem admits a polynomial-time algorithm for semicomplete digraphs, while its complexity remains open for general split digraphs. It is polynomial for semicomplete split digraphs (see Section \ref{sec:remarks}).\\


The main result of this paper provides an affirmative answer to a question by Bang-Jensen and Wang \cite{bang2025strong}.
\begin{theorem}\label{split digraph:main theorem 1}
     Every 6-strong split digraph is 2-linked.
\end{theorem}

Theorem \ref{split digraph:main theorem 1} is an immediate consequence of the following  result concerning  local connectivities of split digraphs.

\begin{theorem}\label{split digraph:theorem}
     Let \(D=(V_1,V_2;A)\) be a split digraph and \(s_1,s_2,t_1,t_2\) be distinct vertices of \(D\). If \(D-\{s_{3-i},t_{3-i}\}\) has three internally disjoint \((s_i,t_i)\)-paths and \(D-\{s_{i},t_{i}\}\) has four internally disjoint \((s_{3-i},t_{3-i})\)-paths for \(i=1\) or \(2\), then \(D\) has a pair of disjoint \((s_1,t_1)\)- and \((s_2,t_2)\)-paths.
\end{theorem}

The bound on local connectivity in Theorem~\ref{split digraph:theorem} is tight, as shown in Proposition~\ref{counterexample-split}.
This illustrates a difference between split digraphs and semicomplete digraphs for which Bang-Jensen proved the following sufficient condition (which is also tight)

\begin{theorem}\label{thm:SD2link}\cite{bang19882}
     Let \(D=(V_1,V_2;A)\) be a semicomplete digraph and \(s_1,s_2,t_1,t_2\) be distinct vertices of \(D\). If \(D-\{s_{3-i},t_{3-i}\}\) has three internally disjoint \((s_i,t_i)\)-paths and \(D-\{s_{i},t_{i}\}\) has two internally disjoint \((s_{3-i},t_{3-i})\)-paths for \(i=1\) or \(2\), then \(D\) has a pair of disjoint \((s_1,t_1)\)- and \((s_2,t_2)\)-paths.
\end{theorem}

As remarked in \cite{bang2025strong}, using the same approach as was used to prove Theorem \ref{thm:SD2link} in \cite{bang19882}, it is straightforward to verify that every 6-strong semicomplete split digraph is 2-linked (see the proof of Theorem \ref{theorem:6-strong semicomplete multipartite} ).
 In this paper, we improve this bound by providing a tight bound on the connectivity for the 2-linkage problem in semicomplete split digraphs.

\begin{theorem}\label{semicomplete split digraph: main theorem 2}
    Every 5-strong semicomplete split digraph is 2-linked.
\end{theorem}

{Since each semicomplete digraph is also a semicomplete split digraph, it follows from Theorem \ref{5-strong semicomplete digraph is 2-linked} that the bound in Theorem \ref{semicomplete split digraph: main theorem 2} is also tight.}

As in the case of  general split digraphs, again we prove a more general result concerning  local connectivity of semicomplete split digraphs by proving  the following result. We have not been able to verify whether this is best possible in terms of local connectivities, or whether the bounds in Theorem \ref{thm:SD2link} also hold for semicomplete split digraphs.

\begin{theorem}\label{semicomplete split digraph:theorem}
     Let \(D=(V_1,V_2;A)\) be a semicomplete split digraph and \(s_1,s_2,t_1,t_2\) be distinct vertices of \(D\). If both \(D-\{s_{3-i},t_{3-i}\}\) and \(D-\{s_{i},t_{i}\}\) has three internally disjoint \((s_{3-i},t_{3-i})\)-paths for \(i=1\) or \(2\), then \(D\) has a pair of disjoint \((s_1,t_1)\)- and \((s_2,t_2)\)-paths.
\end{theorem}

A digraph is {\bf semicomplete multipartite} if it is obtained
from a complete multipartite graph by replacing every edge by an arc or a pair of opposite arcs. It is easy to observe that semicomplete split digraphs are exactly the special class of semicomplete multipartite digraphs, where one partition has order \(|V_1|\) and the remaining partitions all have order \(1\). We establish the following result for semicomplete multipartite digraphs.

\begin{theorem}\label{theorem:6-strong semicomplete multipartite}
    Every 6-strong semicomplete multipartite digraph is 2-linked.
\end{theorem}

{Theorem \ref{5-strong semicomplete digraph is 2-linked} has been generalized to two larger classes of digraphs, in particular to} quasi-transitive digraphs and locally semicomplete digraphs.
A digraph \(D\) is {\bf quasi-transitive} if, for every triple \(x,y,z\) of distinct vertices of \(D\) such that \(xy\) and \(yz\) are arcs of \(D\), there is at least one arc between \(x\) and \(z\). And a digraph \(D\) is {\bf locally semicomplete} if, for every vertex \(x\) in \(D\), both the in-neighbours and out-neighbours of \(x\) induce semicomplete digraphs. Clearly, {each semicomplete digraph is both quasi-transitive and locally semicomplete}. In \cite{bang1999linkages}, Bang-Jensen proved that every 5-strong quasi-transitive {digraph is 2-linked, and conjectured that this also holds for locally semicomplete {digraphs}. This conjecture was confirmed} by Bang-Jensen, Christiansen, and Maddaloni in \cite{bang2017disjoint}. 


{While we have so far focused on the tight connectivity bound for the 2-linkage problem, it is worth noting that tight connectivity bounds for the $k$-linkage problem with general $k$ have also been established for tournaments. Thomassen \cite{thomassen1984connectivity} was the first to prove that there exists a function $f(k)$ such that every $f(k)$-strong {tournament} is $k$-linked. Subsequently, a series of improvements on the connectivity bound have been obtained {for tournaments}, see~\cite{kuhn2014proof,pokrovskiy2015highly,Meng2021improved,bang2022every,ZHOU2023113351,chen2025new,ZHOU2026114700,GIRAO2019251}. In particular, regarding the tightness of the connectivity bound, Pokrovskiy proposed the following conjecture.}

\begin{conjecture} \cite{pokrovskiy2015highly}\label{pokrovskiy2015highly}
{There exists an integer \(f(k)\) such that every \(2k\)-strong tournament with minimum semi-degree at least \(f(k)\) is \(k\)-linked.}
\end{conjecture}

{In 2021, Girão, Popielarz, and Snyder \cite{girao20212} showed that every $(2k+1)$-strong tournament with minimum out-degree at least $Ck^{31}$ is $k$-linked for some constant $C$. In 2026, Zhou and Yan \cite{zhou2026counterexamples} disproved Conjecture \ref{pokrovskiy2015highly}. By combining these two results, it follows that $2k+1$ is the tight connectivity bound for the $k$-linkage problem in tournaments with sufficiently large minimum out-degree. In \cite{girao20212}, Girão, Popielarz, and Snyder further conjectured that a minimum out-degree of $O(k)$ is already sufficient. However, Zhou and Yan \cite{zhou2025proof} gave a negative answer to this conjecture and proved that every $(2k+1)$-strong semicomplete digraph with minimum out-degree at least $7k^2+36k$ is $k$-linked. This result also yields a tight minimum out-degree bound up to a constant factor.
In fact, we believe that an analogous result holds for split digraphs, (see Problem \ref{probelm-split:(2k+1) and ck^2}).

The paper is organized as follows. Section \ref{section 2} introduces {basic terminology, preliminary lemmas, and the counterexamples illustrating the tightness of Theorem \ref{split digraph:theorem}. 
The proofs of Theorem \ref{split digraph:theorem} and Theorem \ref{semicomplete split digraph:theorem} are presented in Section \ref{section 3} and Section \ref{section 4}, respectively}. The proof of Theorem \ref{theorem:6-strong semicomplete multipartite} is given in Section \ref{section 5}. Finally, we list some open problems in Section \ref{sec:remarks}.

\section{Terminology and Preliminaries}\label{section 2}


Notation not specified in this section is consistent with that in \cite{bang2008digraphs}. For an integer \(n\), \([n]\) denotes the set \(\{1,2,\dots,n\}\). Given a digraph \(D = (V, A)\) with vertex set \(V\) and arc set \(A\), the order of \(D\) is denoted by \(|D|\). All digraphs considered herein are simple, i.e., without loops and multiple arcs.

For an arc \((u,v)\), abbreviated as \(uv\), we say that \(u\) {\bf dominates} \(v\) (\(v\) {\bf is dominated by} \(u\)) and sometimes write \(u \to v\) to emphasize the direction of the arc. For vertex sets \(X,Y\subseteq V(D)\), define \((X,Y)_D = \{xy \in A(D) : x \in X,\, y \in Y\}\). For a vertex \(x\in V(D)\), let
\[
N_D^+(x) = \{y \mid xy \in A(D)\}, \quad d_D^+(x) = |N_D^+(x)|
\]
denote the {\bf out-neighborhood} and {\bf out-degree} of \(x\), respectively; the {\bf in-neighborhood} \(N_D^-(x)\) and {\bf in-degree} \(d_D^-(x)\) are defined analogously. For \(X\subseteq V(D)\), the digraph \(D\setminus X\) is obtained by deleting \(X\) and all arcs incident to \(X\). For disjoint \(X,Y\subseteq V(D)\):
\begin{itemize}
\item \(X \to Y\) means every vertex in \(X\) dominates every vertex in \(Y\);
\item \(X \Rightarrow Y\) means \((Y,X)_D = \emptyset\);
\item \(X \mapsto Y\) means both \(X\to Y\) and \(X\Rightarrow Y\) hold.
\end{itemize}

Let \(P = x_1x_2\cdots x_k\) be an \((x_1,x_k)\)-path. For \(1 < i \le k\), the {\bf predecessor} of \(x_i\) on \(P\) is \(x_{i-1}\); for \(1\le i < k\), the {\bf successor} of \(x_i\) on \(P\) is \(x_{i+1}\). The {\bf length} of \(P\) is the number of arcs it contains. For \(i<j\), define \(P[x_i,x_j]\) as the subpath \(x_ix_{i+1}\cdots x_j\) and \(P[x_i,x_j)\) as \(x_ix_{i+1}\cdots x_{j-1}\). An \((x,y)\)-path \(P\) is {\bf minimal} if no shorter \((x,y)\)-path exists in the induced subdigraph \(D\langle V(P)\rangle\). For \(X,Y\subseteq V(D)\), an \((x,y)\)-path \(P\) is an {\bf \((X,Y)\)-path} if \(x\in X\), \(y\in Y\), and \(V(P)\cap(X\cup Y)=\{x,y\}\). 

Let \(Q = y_1y_2\cdots y_t\) be another path. The paths \(P\) and \(Q\) are {\bf disjoint} if \(V(P)\cap V(Q)=\emptyset\), and {\bf internally disjoint} if
\[
\{x_2,\dots,x_{k-1}\}\cap V(Q) = \emptyset \quad \text{and} \quad V(P)\cap \{y_2,\dots,y_{t-1}\}=\emptyset.
\]
If \(x_k = y_1\), we denote by \(P\circ Q\) the concatenated path \(x_1x_2\cdots x_ky_2\cdots y_t\).

For a digraph \(D=(V,A)\) and vertices \(x,y\in V\), let \(\kappa(x,y)\) denote the maximum number of internally disjoint \((x,y)\)-paths in \(D\); we refer to \(\kappa(x,y)\) as the {\bf local strong connectivity} from \(x\) to \(y\). The following consequence of Menger's theorem will be used repeatedly throughout the proofs in this paper.

\begin{theorem}\label{Menger}\cite{MengerZurAK}(Menger) Let \( D = (V, A) \) be a digraph. Then \( D \) is \( k \)-strong if and only if \( |V(D)| \geq k + 1 \) and \( D \) contains \( k \) internally disjoint \( (s,t) \)-paths for every choice of distinct vertices \( s,t \in V \).
\end{theorem}


Next, we present a definition to simplify the exposition in subsequent proofs.

\begin{Definition}\label{Definition:good}
    Let \(D\) be a digraph and \(s_1, s_2, t_1, t_2\) be distinct vertices of \(D\). We say that the tuple \((D, s_1, t_1, s_2, t_2)\) is {\bf good} if \(D\) contains two disjoint paths \(P_1\) and \(P_2\) such that \(P_i\) is an \((s_i, t_i)\)-path for each \(i \in \inttwo\).
\end{Definition}

For a pair $s,t$ of distinct vertices of a digraph $D=(V,A)$, a set $S \subseteq V(D) - \{s,t\}$ is an {\bf $(s,t)$-separator} if $D - S$ has no $(s,t)$-paths.
{Proposition \ref{counterexample-split} implies that the local connectivity bounds \( 3 \) and \( 4 \) in Theorem \ref{split digraph:theorem} are tight.

\begin{proposition}\label{counterexample-split}
    There exists a split digraph \(D=(V_1,V_2;A)\) and four vertices \(s_1,t_1,s_2,t_2\in V(D)\) such that \( \kappa_{D\setminus\{s_2,t_2\}}(s_1,t_1)= 3 \) and \( \kappa_{D \setminus \{s_1,t_1\}}(s_2,t_2) = 3 \), but \(D\) contains no good tuple \((D, s_1, t_1, s_2, t_2)\).
\end{proposition} }

\begin{proof}
We construct a digraph \(D\) as follows (see Figure \ref{Counterexample/Counterexample-split-(3,3)}). Let $V(D)=X\cup Y\cup Z$, where \(X=\{x_1,x_2,x_3\}, Y=\{y_1,y_2,y_3\}, Z=\{z_1,z_2,z_3\}\). And the arc set \(A(D)\) consists of the following arcs.
    \begin{itemize}
        \item \(\{y_1y_2, y_2y_3, y_3y_1\} \subseteq A(D)\);
        \item \(\{s_1s_2,t_1s_2,s_1t_2,t_1t_2,t_1s_1,t_2s_2\} \subseteq  A(D)\);
        \item \(\{s_1\} \mapsto X, Z \mapsto \{t_1\}, \{t_1\} \mapsto Y, Y \mapsto \{s_1\}\);
        \item The arcs of the paths $x_i\to y_i\to z_i$ for $i\in [3]$;
        \item \(X \cup Y \mapsto \{s_2\}, \{s_2\} \mapsto Z, \{t_2\} \mapsto Y \cup Z, X \mapsto \{t_2\}\);
        \item The arcs of the paths $z_i\to y_{i+1} \to x_{i+2}$ for $i\in [3]$, with subscripts modulo 3.
    \end{itemize}

 \begin{figure}[htbp] 
    \centering 
    \includegraphics[width=0.45\textwidth]{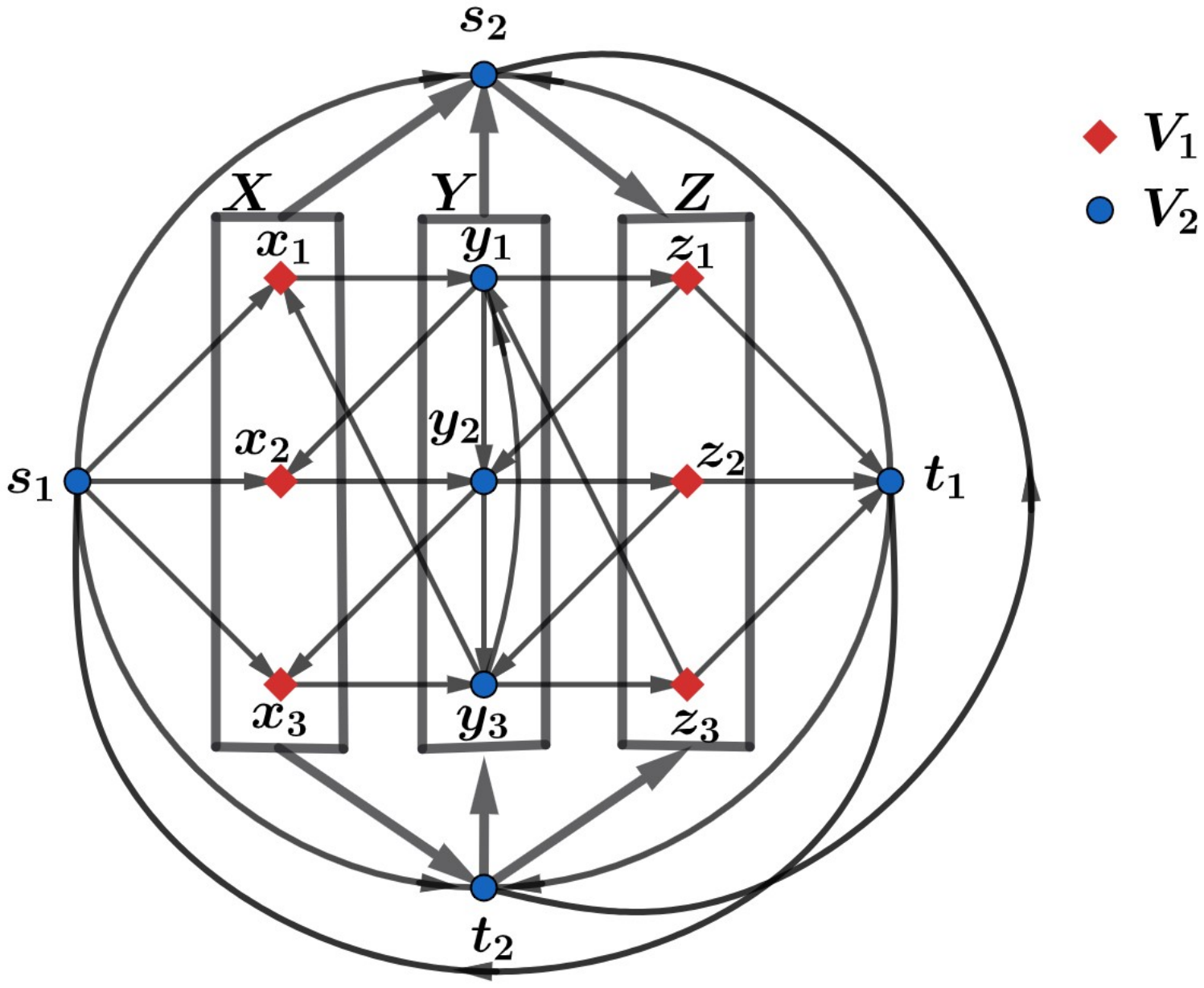}
  \caption{Red diamonds represent vertices in \(V_1\), blue  vertices represent vertices in \(V_2\). Let
$P_i = s_1 \to x_i \to y_i \to z_i \to t_1$
and
$Q_i = s_2 \to z_i \to y_{i+1} \to x_{i+2} \to t_2$,
where subscripts in $Q_i$ are taken modulo $3$.
Combined with $d_{D-\{s_2,t_2\}}^+(s_1)=3$ and $d_{D-\{s_1,t_1\}}^+(s_2)=3$,
this yields
$\kappa_{D\setminus\{s_2,t_2\}}(s_1,t_1)=3$
and
$\kappa_{D\setminus\{s_1,t_1\}}(s_2,t_2)=3$.} 
    \label{Counterexample/Counterexample-split-(3,3)} 
\end{figure}
    
    It is easy to check $D$ is a split digraph $D=(V_1,V_2;A(D))$, by letting \(V_1=X \cup Z\) and \(V_2=\{s_1,s_2,t_1,t_2\} \cup Y\). Furthermore, \(D\) satisfies  $\kappa_{D-\{s_2,t_2\}}(s_1,t_1)= 3$ and \( \kappa_{D-\{s_1,t_1\}}(s_2,t_2) = 3\). Next, we verify the following conclusion regarding $(s_1,t_1)$-separators in $D - \{s_2, t_2\}$.

   \begin{claim}\label{Claim:Counterexample-split-digraph-1}
    For each $i \in \intthree$, the set $\{z_i, y_{i+1}, x_{i+2}\}$
   with subscripts modulo $3$ is an $(s_1,t_1)$-separator in $D - \{s_2, t_2\}$.
   \end{claim}

    \begin{proof}
    Suppose for contradiction that there exists some $i \in [3]$ such that $\{z_i, y_{i+1}, x_{i+2}\}$ (subscripts modulo $3$) is not an $(s_1,t_1)$-separator in $D - \{s_2, t_2\}$. By definition, this means there exists an $(s_1,t_1)$-path in $D - \{s_2, t_2, z_i, y_{i+1}, x_{i+2}\}$. As illustrated in Figure \ref{Counterexample/Counterexample-split-(3,3)}, $N_{D - \{s_2, t_2, z_i, y_{i+1}, x_{i+2}\}}^+(s_1) = \{x_i, x_{i+1}\}$. Note that 
    \begin{itemize}
        \item $N_{D - \{s_2, t_2, z_i, y_{i+1}, x_{i+2}\}}^+(x_{i+1}) = \emptyset$;
        \item $N_{D - \{s_2, t_2, z_i, y_{i+1}, x_{i+2}\}}^+(x_i) = \{y_i\}$;
        \item $N_{D - \{s_2, t_2, z_i, y_{i+1}, x_{i+2}\}}^+(y_i) = \{x_{i+1}\}$. 
    \end{itemize}
    Any path starting at $s_1$ in $D - \{s_2, t_2, z_i, y_{i+1}, x_{i+2}\}$ thus terminates at $x_{i+1}$ or $y_i$, with no continuation to $t_1$. Hence, no $(s_1,t_1)$-path exists in this digraph, a contradiction to assumption.
\end{proof}

We also have the following analogous result.

  \begin{claim}\label{Claim:Counterexample-split-digraph-2}
   For each $i \in \intthree$, the set $\{z_i, y_{i+1}, y_{i+2}\}$
   with subscripts modulo $3$ is an $(s_1,t_1)$-separator in $D - \{s_2, t_2\}$.
   \end{claim}

   \begin{proof}
      Suppose for contradiction that there exists some $i \in [3]$ such that $\{z_i, y_{i+1}, y_{i+2}\}$ (subscripts modulo $3$) is not an $(s_1,t_1)$-separator in $D - \{s_2, t_2\}$. By definition, this means there exists an $(s_1,t_1)$-path in $D - \{s_2, t_2, z_i, y_{i+1}, y_{i+2}\}$. As illustrated in Figure \ref{Counterexample/Counterexample-split-(3,3)}, $N_{D - \{s_2, t_2, z_i, y_{i+1}, y_{i+2}\}}^-(t_1) = \{z_{i+1},z_{i+2}\}$. Note that 
     
    \begin{itemize}
        \item $N_{D - \{s_2, t_2, z_i, y_{i+1}, y_{i+2}\}}^-(z_{i+1}) = \emptyset$;
        \item $N_{D - \{s_2, t_2, z_i, y_{i+1}, y_{i+2}\}}^-(z_{i+2}) = \emptyset$.
    \end{itemize}
    Hence, no $(s_1,t_1)$-path exists in $D - \{s_2, t_2, z_i, y_{i+1}, y_{i+2}\}$, contradicting assumption.
   \end{proof}

  As can easily be checked by inspecting Figure~\ref{Counterexample/Counterexample-split-(3,3)}, every $(s_2,t_2)$-path $P$ in $D$ satisfies the following: there exists some $i \in \intthree$ such that either $\{z_i, y_{i+1}, x_{i+2}\} \subseteq V(P)$ or $\{z_i, y_{i+1}, y_{i+2}\} \subseteq V(P)$ (with subscripts taken modulo $3$). Combining this observation with Claims \ref{Claim:Counterexample-split-digraph-1} and \ref{Claim:Counterexample-split-digraph-2}, we conclude that there is no $(s_2,t_2)$-path in $D$ that avoids intersecting some $(s_1,t_1)$-path.
  Thus, \(D\) is a split digraph with the prescribed local connectivities, but the tuple \((D, s_1, t_1, s_2, t_2)\) is not good.
    
\end{proof}

\section{Proof of Theorem \ref{split digraph:theorem}}\label{section 3}


{Before proving this conclusion, we first give the sketch of the proof by contradiction, i.e., we assume that \((D, s_1, t_1, s_2, t_2)\) is not good. By contradiction and the pigeonhole principle, we extract four important subpaths \(Q_{i'}^*, Q_{j'}^*, Q_{k'}^*, Q_{l'}^*\) from the internally disjoint \((s_2,t_2)\)-paths, construct a short \((s_1,t_1)\)-path \(P\), and fix an \((s_1,t_1)\)-path \(P_{k_0}\). By our hypothesis, any new \((s_2,t_2)\)-path incorporating \(Q_{i'}^*, Q_{j'}^*, Q_{k'}^*, Q_{l'}^*\) (along with other necessary path segments) must intersect \(P\), which enables us to fully characterize \(P_{k_0}\) and reach a contradiction from either arc direction between two vertices in \(V_2\).}\\

\noindent \textbf{Proof of Theorem \ref{split digraph:theorem}.}
   Without loss of generality, we may assume that \(i = 1\), i.e., \(D-\{s_{2},t_{2}\}\) has three internally disjoint \((s_1,t_1)\)-paths and \(D-\{s_{1},t_{1}\}\) has four internally disjoint \((s_{2},t_{2})\)-paths.   Let \(P_1, P_2, P_3\) be three internally disjoint minimal \((s_1,t_1)\)-paths in \(D - \{s_2,t_2\}\), and let \(Q_1,Q_2,Q_3,Q_4\) be {four} minimal internally disjoint \((s_2,t_2)\)-paths in \(D - \{s_1,t_1\}\).

   Suppose {to the contrary} that 
   \begin{equation}\label{assume}
     \text{\((D, s_1, t_1, s_2, t_2)\) is not good.}  
   \end{equation}
   {Then each \(Q_i\) (for \(i \in \intfour\)) intersects all \(P_j\) (for \(j \in \intthree\)).} Otherwise, there exists a pair of disjoint \((s_1,t_1)\)-path and \((s_2,t_2)\)-path, a contradiction. For each \(i \in \intfour\), denote by \(z_i\) and \(w_i\) the first and last vertices on \(Q_i\) intersecting \(\bigcup_{j\in [3]} V(P_j)\), respectively. By the pigeonhole principle, there exist distinct \(i',j' \in [4]\) and \(i_0 \in [3]\) such that \(z_{i'}\) and \(z_{j'}\) lie on \(P_{i_0}\). Without loss of generality, assume \(z_{i'}\) precedes \(z_{j'}\) on \(P_{i_0}\), i.e., \(|P_{i_0}[s_1,z_{i'}]| < |P_{i_0}[s_1,z_{j'}]|\). Symmetrically, there exist distinct \(k',l' \in [4]\) and \(j_0 \in [3]\) such that \(w_{k'}\) and \(w_{l'}\) lie on \(P_{j_0}\); assume \(w_{k'}\) precedes \(w_{l'}\) on \(P_{j_0}\). Note that $i_0$ and $j_0$ are not necessarily distinct. We define four critical subpaths for subsequent path adjustments: \(Q_{i'}^* = Q_{i'}[s_2,z_{i'}]\), \(Q_{j'}^* = Q_{j'}[s_2,z_{j'}]\), \(Q_{k'}^* = Q_{k'}[w_{k'},t_2]\), and \(Q_{l'}^* = Q_{l'}[w_{l'},t_2]\). Let \(b\) be the predecessor of \(t_1\) on \(P_{i_0}\), and let \(a\) be the predecessor of \(b\) on \(P_{i_0}\); let \(c\) be the successor of \(s_1\) on \(P_{j_0}\), and let \(d\) be the successor of \(c\) on \(P_{j_0}\) (see Figure~\ref{case4-1}).

   \begin{figure}[htbp] 
    \centering 
    \includegraphics[width=0.45\textwidth]{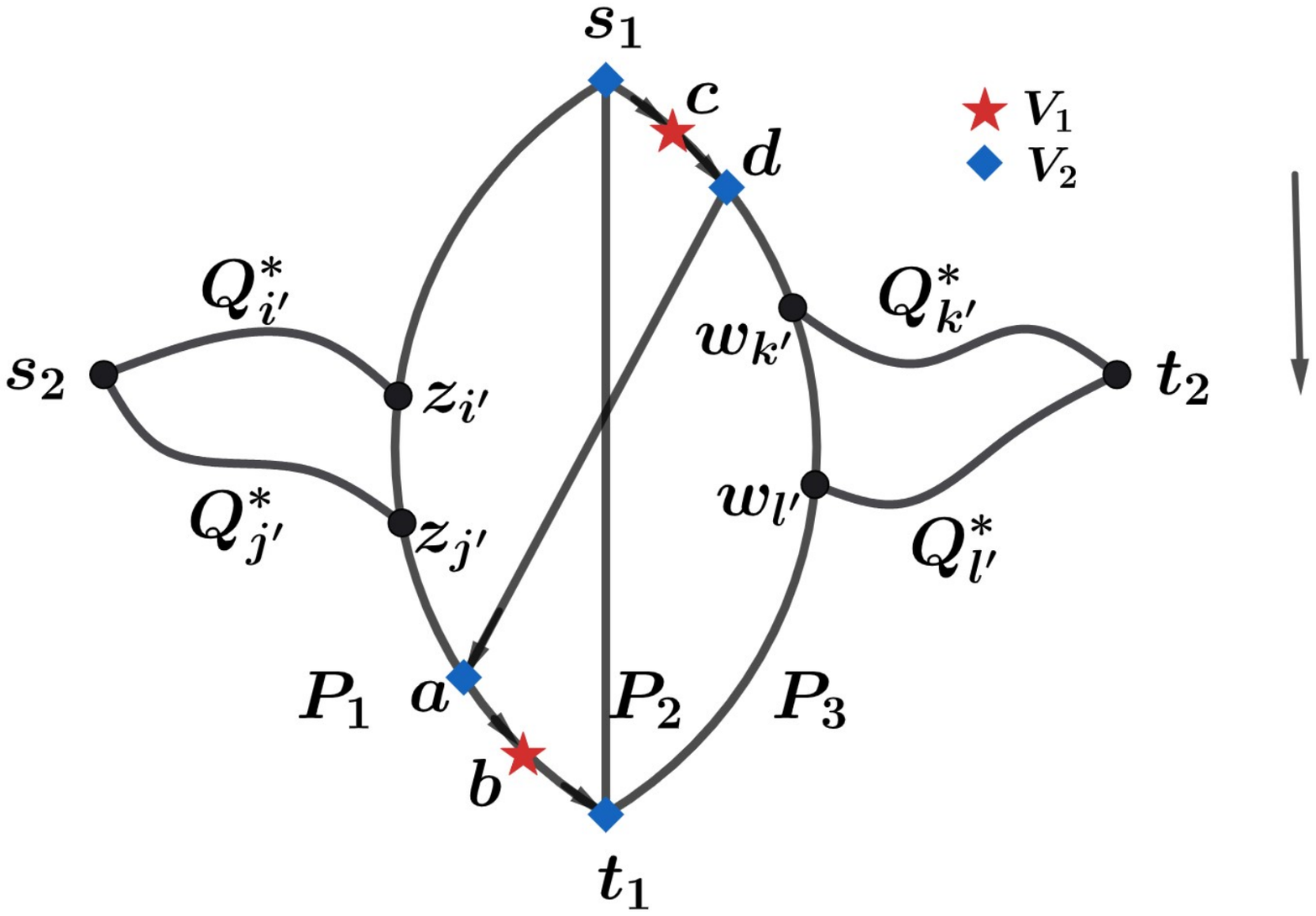}
  \caption{Example with \(i_0 = 1\) and \(j_0 = 3\). Red solid stars represent vertices in \(V_1\), blue diamonds represent vertices in \(V_2\), and black vertices are not fixed to be in either $V_1$ or $V_2$. The paths \(P_1,P_2,P_3\) are directed from top to bottom. This coloring rule is used for all following figures (without further mentioning).}
    \label{case4-1} 
\end{figure}

To obtain a new $(s_1,t_1)$-path, we need to prove that $da \in A(D)$,  as required by Claim~\ref{case4:claim1}. To this end, we first establish Claim~\ref{claim3.1}.
\begin{claim}\label{claim3.1}
    $c$ is non-adjacent to $\{a,b\}$ and $b$ is non-adjacent to $d$.
\end{claim}
\begin{proof}
    Suppose that \( c \) is adjacent to either \( a \) or \( b \). According to the adjacency between $c$ and 
    $\{a,b\}$, we consider the two cases separately.

Case 1: \(a \to c\) or \(b \to c\). Let \(f = a\) if \(a \to c\), and \(f = b\) if \(b \to c\). In either case, \(f \to c\). The path
\[
Q = Q_{i'}^* \circ P_{i_0}[z_{i'},f] \circ \{fc\} \circ P_{j_0}[c,w_{k'}] \circ Q_{k'}^*
\]
is an \((s_2,t_2)\)-path that is disjoint from \(P_i\) for all \(i \in [3] \setminus \{i_0,j_0\}\), contradicting \eqref{assume}.

Case 2: \(c \to a\) or \(c \to b\). Then either $s_1cabt_1$ or $s_1cbt_1$ is a path and as it has at most 3 internal vertices, it avoids at least one of the paths $Q_1,\ldots{},Q_4$, contradicting \eqref{assume}.\\

By a similar argument as in Case 1, $D$ cannot contain the arc $bd$ and since $D$ has no $(s_1,t_1)$-path of length at most 4 avoiding $s_2,t_2$ we see that $D$ cannot have an arc from $d$ to $b$. Thus the claim is proved.

\end{proof}

\begin{claim}\label{case4:claim1}
\(\{b,c\} \subseteq V_1\) and \(\{a, d\}\subseteq V_2\). Moreover, \(da \in A(D)\).
\end{claim}

\begin{proof}
Suppose that \( c \in V_2 \). Then \( c \) is adjacent to either \( a \) or \( b \), {which contradicts Claim \ref{claim3.1}. Thus, \( c \in V_1 \). Symmetrically, we can verify that \( b \in V_1 \). Consequently, it follows directly from \(a\to b\) and \(c\to d\) that \(a, d \in V_2\). To verify \(d\to a\), we assume $a\to d$, so \(Q_{i'}^* \circ P_{i_0}[z_{i'},a] \circ \{ad\} \circ P_{j_0}[d,w_{l'}] \circ Q_{l'}^*\) is an $(s_2,y_2)$-path, which is disjoint with \(P_i\) for \(i \in \intthree \setminus \{i_0,j_0\}\),  contradicting (\ref{assume}).}
\end{proof}

 By the claims above, $D$ contains the $(s_1,t_1)$-path 
 $P = s_1cdabt_1$. As $Q_1,\ldots{},Q_4$ are internally disjoint it follows from (\ref{assume}), each of them contain exactly one vertex from $\{a,b,c,d\}$ and only one of these vertices is on some $Q_i$. 
 Without loss of generality, we assume that \(a\in V(Q_1)\), \(b\in V(Q_2)\), \(c\in V(Q_3)\), and \(d\in V(Q_4)\).

Let $k_0$ be an index in $\intthree \setminus \{i_0, j_0\}$. We first argue that each of the four subpaths $Q_1(a,t_2]$, $Q_2(b,t_2]$, $Q_3[s_2,c)$, and $Q_4[s_2,d)$ intersect $P_{k_0}$ in at least one vertex. Suppose that \(V( Q_1(a, t_2])\cap  V(P_{k_0}) =\emptyset\). Then the path \( Q = Q_{i'}^* \circ P_{i_0}[z_{i'}, a] \circ Q_1[a, t_2] \) is disjoint from $P_{k_0}$, contradicting (\ref{assume}) (See Figure \ref{case4-figure2-(a)}).  Analogously, \(V( Q_3[s_2, c) )\cap V( P_{k_0})\neq \emptyset \), since otherwise, the path \( Q' = Q_3[s_2, c] \circ P_{j_0}[c, w_{k'}] \circ Q_{k'}^* \), is disjoint from \( P_{k_0} \), contradicting (\ref{assume}) again (See Figure \ref{case4-figure2-(b)}). The remaining two cases follow analogously.

\begin{figure}[htbp]  
    \centering
    \begin{subfigure}[b]{0.455\linewidth}
        \centering
        \includegraphics[width=\linewidth]{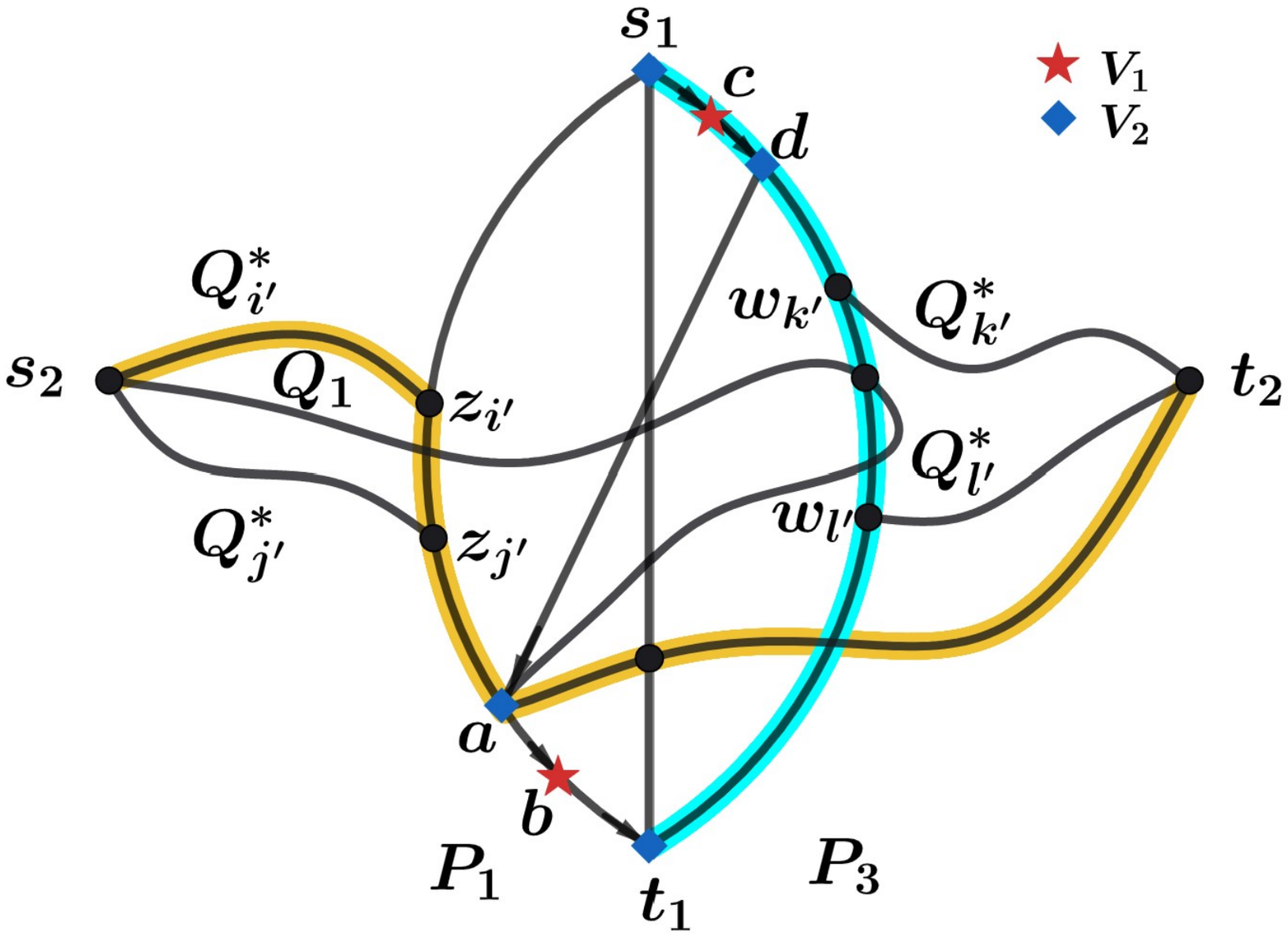}  
        \caption{}
        \label{case4-figure2-(a)}
    \end{subfigure}
    \hfill  
    \begin{subfigure}[b]{0.455\linewidth}
        \centering
        \includegraphics[width=\linewidth]{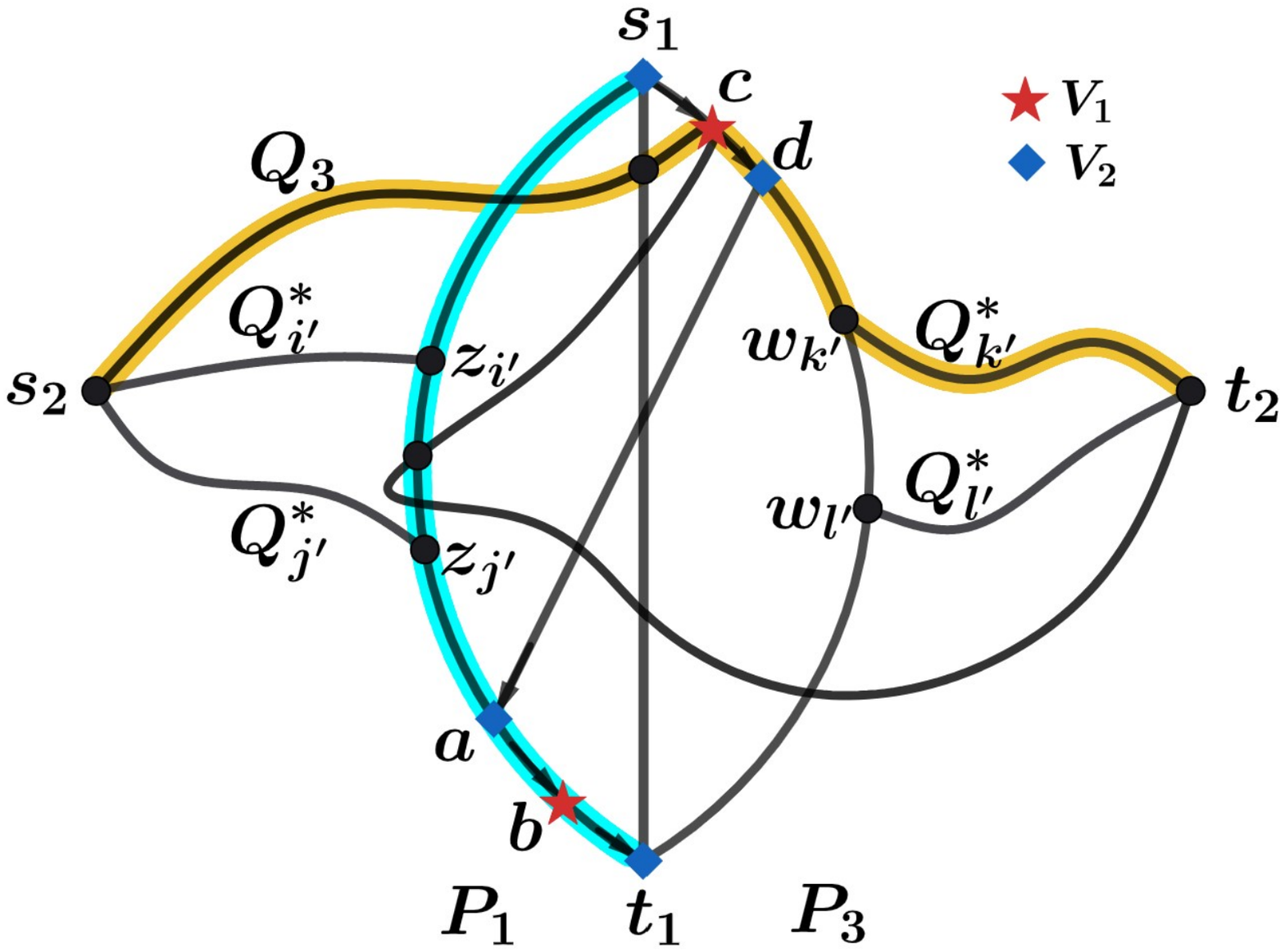}
        \caption{}
        \label{case4-figure2-(b)}
    \end{subfigure}
    \hfill 
 \caption{Example with $i_0=1$ and $j_0=3$. (a) illustrates the case when $Q_1[a,t_2]$ and $P_3$ are disjoint; (b) illustrates the case when $Q_3[s_2,c]$ and $P_1$ are disjoint. }
    \label{case4-2} 
    
\end{figure}

Choose four vertices $x_1, x_2, x_3, x_4$ satisfying the following properties.
\begin{itemize}
    \item  \( x_1 \in V(Q_1(a, t_2])\cap V( P_{k_0})\)  such that \( |P_{k_0}[x_1,t_1]| \) is minimized;
    \item \( x_2 \in V(Q_2(b, t_2])\cap V( P_{k_0}) \) such that  \( |P_{k_0}[x_2,t_1]| \) is minimized;
    \item  \( x_3 \in V(Q_3[s_2, c))\cap  V(P_{k_0})\) such that  \( |P_{k_0}[s_1,x_3]| \) is minimized;  
    \item  \( x_4\in  V(Q_4[s_2, d))\cap  V(P_{k_0}) \) such that \( |P_{k_0}[s_1,x_4]| \) is minimized.
\end{itemize}
Furthermore, we can determine the relative order of the vertices $x_1, x_2, x_3, x_4$ along $P_{k_0}$. 
\begin{claim}\label{case4:claim4}
$\max\{|P_{k_0}[s_1,x_1]|, |P_{k_0}[s_1,x_2]|\}< \min\{|P_{k_0}[s_1,x_3]|, |P_{k_0}[s_1,x_4]|\}$.
\end{claim}

\begin{proof}
Without loss of generality, we may assume that $|P_{k_0}[s_1,x_1]|=\max\{|P_{k_0}[s_1,x_1]|, |P_{k_0}[s_1,x_2]|\}$ and $|P_{k_0}[s_1,x_3]|=\min\{|P_{k_0}[s_1,x_3]|, |P_{k_0}[s_1,x_4]|\}$. If $|P_{k_0}[s_1,x_1]|> |P_{k_0}[s_1,x_3]|$, we then construct the \((s_2,t_2)\)-path 
\(Q = Q_3[s_2, x_3] \circ P_{k_0}[x_3, x_1] \circ Q_i[x_1, t_2]\).
It follows from the definition of the four subpaths above that \(Q\) is disjoint from the \((s_1,t_1)\)-path \(P = s_1cdabt_1\), contradicting (\ref{assume}). Consequently, $\max\{|P_{k_0}[s_1,x_1]|, |P_{k_0}[s_1,x_2]|\}< \min\{|P_{k_0}[s_1,x_3]|, |P_{k_0}[s_1,x_4]|\}$.
\end{proof}

Based on Claim \ref{case4:claim4}, we may assume without loss of generality that $|P_{k_0}[s_1,x_1]|<|P_{k_0}[s_1,x_2]|<|P_{k_0}[s_1,x_3]|<|P_{k_0}[s_1,x_4]|$ (see Figure \ref{case4-3.2}). Clearly, we have $|P_{k_0}| \geq 6$. Furthermore, the subsequent Claim \ref{case4:claim5} confirms that $|P_{k_0}|=6$ and determines the distribution of its internal vertices.

\begin{claim}\label{case4:claim5}
    \(x_{1}, x_{4}\in V_1\) and \(x_{2}, x_{3}\in V_2\). Furthermore, \(P_{k_0} = s_1x_{1}x_{2}x_{3}x_{4}t_1\).
\end{claim} 

\begin{proof}
Our argument relies on the following key observation,
\begin{equation}\label{equation1}
    \text{there is no arc from } P_{k_0}[x_3,t_1) \text{ to } P_{k_0}(s_1,x_2].
\end{equation}
Suppose for a contradiction that such an arc \(uv\) exists, where \(u\in V(P_{k_0}[x_3,t_1))\) and \(v\in V(P_{k_0}(s_1,x_2])\). Define the path
\[
Q = Q_3[s_2, x_3] \circ P_{k_0}[x_3,u] \circ uv \circ P_{k_0}[v,x_2] \circ Q_2[x_2, t_2].
\]
Then \(Q\) is an \((s_2,t_2)\)-path that is disjoint from the \((s_1,t_1)\)-path \(P = s_1cdabt_1\), which contradicts (\ref{assume}).

By the definition of split digraphs, there is no arc within \(V_1\) and there is  at least one arc between every pair of distinct vertices in \(V_2\).  Combining this with \eqref{equation1}, the only possible configuration is that \(x_1, x_4 \in V_1\), \(x_2, x_3 \in V_2\), and \(P_{k_0} = s_1x_1x_2x_3x_4t_1\). This completes the proof.
\end{proof}

\begin{figure}[htbp] 
    \centering 
    \includegraphics[width=0.45\textwidth]{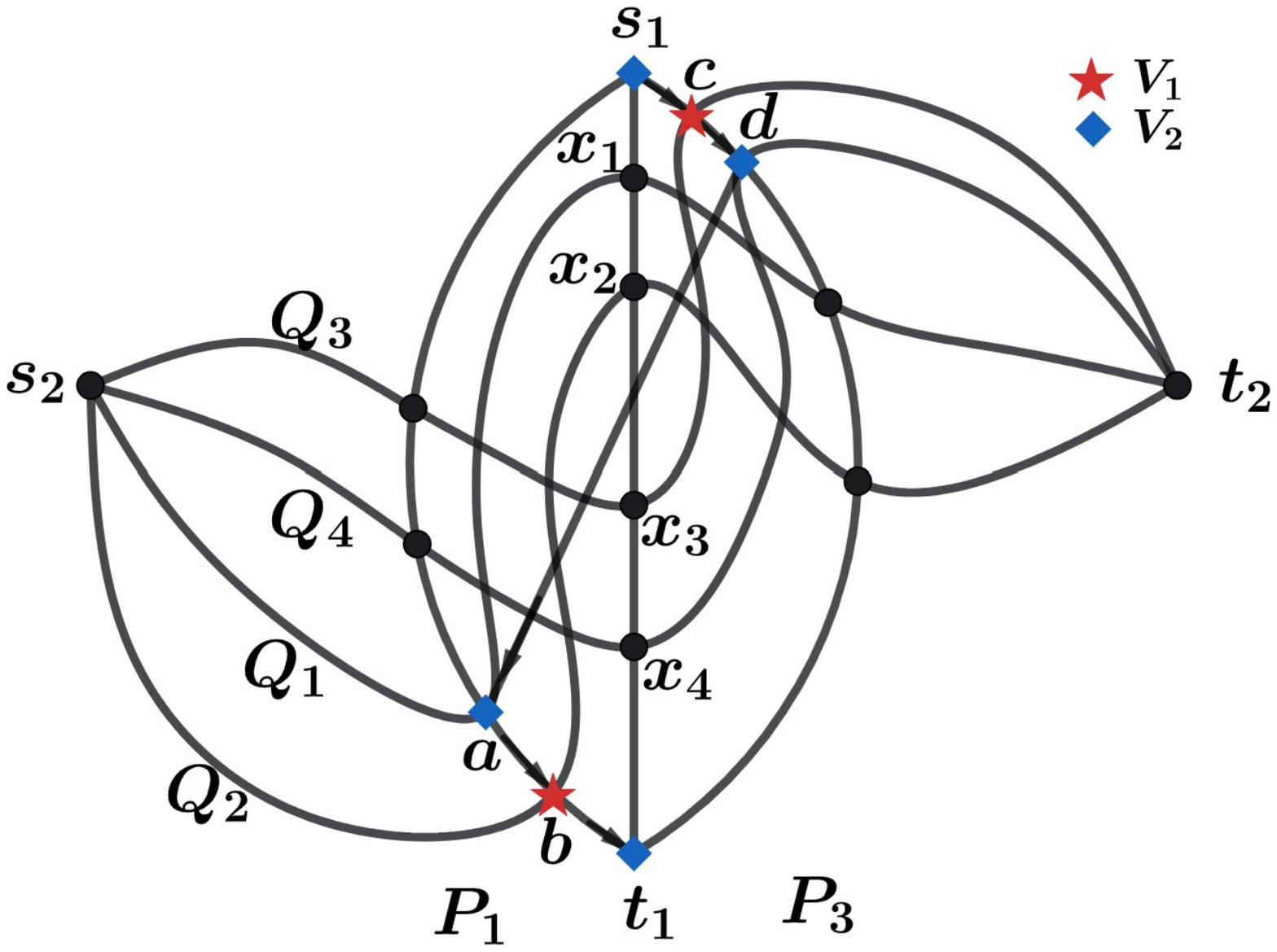}
    \caption{Example with \(i_0=1\), \(j_0=3\).}
    \label{case4-3.2} 
\end{figure}

Now we can complete the proof as follows (see Figure \ref{case4-3.2}). We first conclude from Claims~\ref{case4:claim1} and \ref{case4:claim5} that \(d, x_2 \in V_2\). Suppose \(d  \to x_2\), then the \((s_1,t_1)\)-path \(P_{j_0}[s_1,d] \circ dx_2 \circ P_{k_0}[x_2,t_1]\) is disjoint from the \((s_2,t_2)\)-path \(Q_1\), a contradiction to \eqref{assume}. It follows that \(x_2 \to d \). But then the \((s_1,t_1)\)-path \( P_{k_0}[s_1,x_2] \circ x_2da \circ P_{i_0}[a,t_1]\) is disjoint from \(Q_3\), which yields another contradiction. Therefore, the proof is complete. 

\qed

\section{Proof of Theorem \ref{semicomplete split digraph:theorem}}\label{section 4}

\noindent Before proving Theorem~\ref{semicomplete split digraph:theorem}, we first present a sketch of the proof.
We establish Theorem~\ref{semicomplete split digraph:theorem} by contradiction. Assume that the $(D, s_1, t_1, s_2, t_2)$ is not good. The proof is divided into three cases based on the distributions of the vertices $z_1,z_2,z_3$ and $w_1,w_2,w_3$ on the paths $P_1,P_2,P_3$. As in the proof above, one $(s_1,t_1)$-path $P_{k_0}$ plays a central role.

In \ref{Case 1}, we extract four subpaths and apply the definition of semicomplete split digraphs to fix the orientation of a key arc. A contradiction is then derived by concatenating these subpaths and several key arcs. This reasoning implies that \(P_{k_0}\) contains exactly five vertices for some \(k_0\in \intthree\), and we determine the distribution of its internal vertices over \(V_1\) and \(V_2\). Notably, any orientation of the critical adjacent arc contradicts the assumption. In \ref{Case 2}, we extract two subpaths to fix one arc orientation, then verify that $|V(P_1)|\geq4$. By concatenating subpaths, we construct disjoint $(s_1,t_1)$- and $(s_2,t_2)$-paths, contradicting the assumption.
In \ref{Case 3}, path minimality combined with the structure of semicomplete split digraphs implies $|V(P_i)|=5$ for all $i\in\intthree$, with fixed distributions of internal vertices over $V_1,V_2$. After characterizing the structure of $D$ (see Figure~\ref{semicomplete-split-digraph-Figure5}), we show that any orientation of the critical adjacent arc yields a contradiction.
All cases contradict the assumption, which completes the proof.

\vspace{2mm}

\noindent \textbf{Proof of Theorem \ref{semicomplete split digraph:theorem}.}

   {We proceed by contradiction and suppose that}
   \begin{equation}\label{assume-semicomplete-split}
  {(D, s_1, t_1, s_2, t_2) \text{ is not good}.}
   \end{equation}
  It follows immediately that
   \begin{equation}\label{length-semicomplete-split}
   D \text{ contains no } (s_1,t_1)\text{-path of length at most } 3.
   \end{equation} 
   By hypothesis, \(D - \{s_2, t_2\}\) contains three internally disjoint minimal \((s_1,t_1)\)-paths, denoted \(P_1, P_2, P_3\), and \(D - \{s_1,t_1\}\) contains three internally disjoint minimal \((s_2,t_2)\)-paths, denoted \(Q_1, Q_2, Q_3\).
  Assumption \eqref{assume-semicomplete-split} immediately implies that
  \begin{equation}\label{intersection-semicomplete-split}
   V(P_i) \cap V(Q_j) \neq \emptyset \quad \text{for all } i,j \in [3].
  \end{equation}
  Let \(S = \bigcup_{i=1}^{3}V(P_i) \setminus \{s_1,t_1\}\). For each \(i \in [3]\),  denote by \(z_i\) and \(w_i\) the first and last vertices in \(V(Q_i)\cap S\) along the path $Q_i$, respectively. Further, let \(s_1^1, s_1^2, s_1^3\) be the successors of \(s_1\) on \(P_1, P_2, P_3\), respectively, and \(t_1^1, t_1^2, t_1^3\) be the predecessors of \(t_1\) on \(P_1, P_2, P_3\), respectively.

\begin{claim}\label{Locally semicomplete split:claim 2}
   For each $i\in [3]$, if \(z_i \in V(P_j)\) and \(w_i \in V(P_k)\) for some \(j,k \in [3]\) (possibly \(j=k\)), then \(\{t_1^j, s_1^k\} \subseteq V_1\).
\end{claim}

\begin{proof}
   Suppose that \(\{t_1^j, s_1^k\} \cap V_2 \neq \emptyset\). Then \(t_1^j\) is adjacent to \(s_1^k\) by the definition of semicomplete split digraphs. If \(s_1^k \to t_1^j\), then \(P = s_1s_1^kt_1^jt_1\) is an
   \((s_1,t_1)\)-path of length 3, contradicting \eqref{length-semicomplete-split}.
   Hence, \(t_1^j \to s_1^k\), but then the \((s_2,t_2)\)-path \(Q = Q_j[s_2,z_j] \circ P_j[z_j,t_1^j] \circ t_1^j s_1^k \circ P_k[s_1^k,w_i] \circ Q_i[w_i,t_2]\) is disjoint from \(P_l\) for \(l \in \{1,2,3\} \setminus \{j,k\}\), a contradiction to \eqref{assume-semicomplete-split}.
\end{proof}

The proof below is divided into three cases based on the distributions of the vertices $z_1,z_2,z_3$ and $w_1,w_2,w_3$ on \(P_1,P_2,P_3\).

\vspace{4mm}

\mycase{There exist indices \(i_0,j_0\in [3]\) (possibly \(i_0=j_0\)) such that \(|\{z_1,z_2,z_3\} \cap V(P_{i_0})| \geq 2\) and \(|\{w_1,w_2,w_3\} \cap V(P_{j_0})| \geq 2\).}{Case 1}

Let \(\{z_{i'}, z_{j'}\} \subseteq \{z_1,z_2,z_3\} \cap V(P_{i_0})\) and \(\{w_{k'}, w_{l'}\} \subseteq \{w_1,w_2,w_3\} \cap V(P_{j_0})\) for some \(i',j',k',l' \in [3]\) with \(i' \neq j'\) and \(k' \neq l'\).
{We further assume, without loss of generality, that} \(|P_{i_0}[s_1,z_{i'}]| < |P_{i_0}[s_1,z_{j'}]|\) and \(|P_{j_0}[s_1,w_{k'}]| < |P_{j_0}[s_1,w_{l'}]|\).
We denote \(Q_{i'}^* = Q_{i'}[s_2,z_{i'}]\), \(Q_{j'}^* = Q_{j'}[s_2,z_{j'}]\), \(Q_{k'}^* = Q_{k'}[w_{k'},t_2]\) and \(Q_{l'}^* = Q_{l'}[w_{l'},t_2]\).
For notational convenience, let \(b\) be the predecessor of \(t_1\) on \(P_{i_0}\) (i.e., \(b = t_1^{i_0}\)) and \(a\) be the predecessor of \(b\) on \(P_{i_0}\); let \(c\) be the successor of \(s_1\) on \(P_{j_0}\) (i.e., \(c = s_1^{j_0}\)) and \(d\) be the successor of \(c\) on \(P_{j_0}\) (See Figure \ref{semicomplete-split-digraph-Figure1}).

\begin{figure}[htbp] 
    \centering 
    \includegraphics[width=0.45\textwidth]{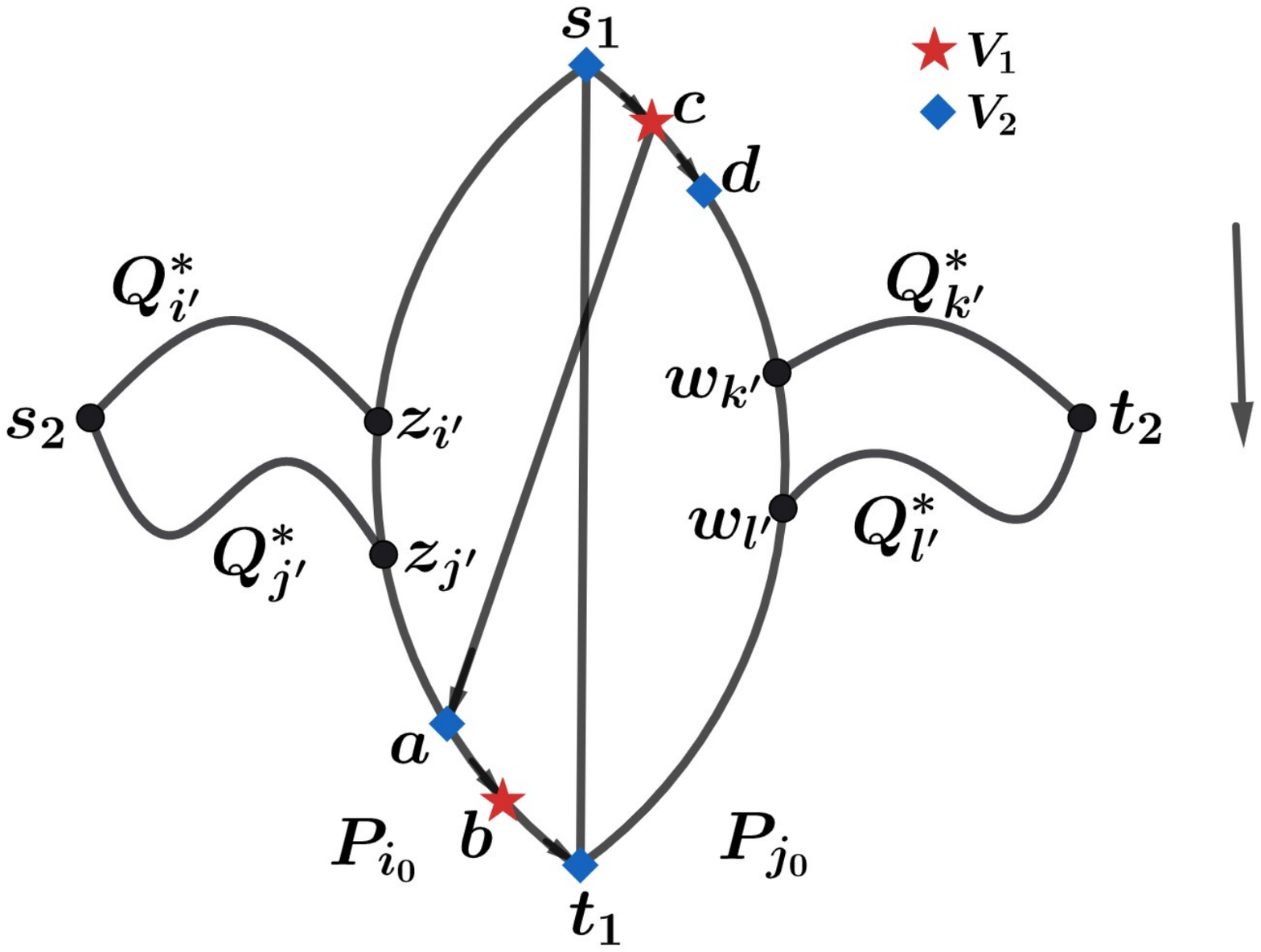}
    \caption{Example with \(i_0 = 1\) and \(j_0 = 3\).
   Red solid stars represent vertices in \(V_1\), blue diamonds represent vertices in \(V_2\), and black vertices are not fixed to be in either $V_1$ or $V_2$. The paths \(P_1,P_2,P_3\) are directed from top to bottom. This coloring rule is used for all following figures (without further mentioning).}
    \label{semicomplete-split-digraph-Figure1}
\end{figure}

To construct a new \((s_1,t_1)\)-path that is disjoint from an \((s_2,t_2)\)-path (either \(Q_i\) for some \(i \in \intthree\) or a newly constructed \((s_2,t_2)\)-path), we first establish Claim \ref{Locally semicomplete split: claim 3}.\\

\begin{claim}\label{Locally semicomplete split: claim 3}
 \(ca \in A(D)\).
\end{claim}

\begin{proof}
Claim \ref{Locally semicomplete split:claim 2} implies that \(b, c \in V_1\), which further yields \(a, d \in V_2\). In particular $a$ and $c$ are adjacent and $b$ and $d$ are adjacent.
If \(a \to c\), then \(Q = Q_{i'}^* \circ P_{i_0}[z_{i'},a] \circ ac \circ P_{j_0}[c,w_{k'}] \circ Q_{k'}^*\) forms an \((s_2,t_2)\)-path that is disjoint from \(P_i\) for all \(i \in [3] \setminus \{i_0,j_0\}\), contradicting \eqref{assume-semicomplete-split}.
Hence it follows that \(ca \in A(D)\). 
\end{proof}

By Claim \ref{Locally semicomplete split: claim 3}, we obtain a new $(s_1, t_1)$-path $P' = s_1 c a b t_1$, which we shall use in the subsequent proof and refer to in the following arguments.\\

It follows from the assumption that $(D,s_1,t_1,s_2,t_2)$ is not good that each of $Q_1,Q_2,Q_3$ contain exactly one of the vertices $a,b,c$, because they are internally disjoint and all intersect $P'$. 
Without loss of generality, we may assume \(a\in V(Q_1)\), \(b\in V(Q_2)\), and \(c\in V(Q_3)\). 

Let \(k_0\) denote an arbitrary index in \(\intthree\setminus\{i_0,j_0\}\). We next analyze the intersections between the subpaths of \(Q_1, Q_2, Q_3\) and \(P_{k_0}\). Suppose {\(V(Q_1[a,t_2])\cap V(P_{k_0})=\emptyset\).} Then the \((s_2,t_2)\)-path \(Q = Q_{i'}^* \circ P_{i_0}[z_{i'},a] \circ Q_1[a,t_2]\) is disjoint from \(P_{k_0}\), contradicting \eqref{assume-semicomplete-split}. Analogously, \(|V(Q_2[b,t_2])\cap V(P_{k_0})|\ge 1\). Moreover, if \(V(Q_3[s_2,c])\cap V(P_{k_0})=\emptyset\), then the \((s_2,t_2)\)-path \(Q' = Q_3[s_2,c] \circ P_{j_0}[c,w_{k'}] \circ Q_{k'}^*\) is disjoint from \(P_{k_0}\), a contradiction to \eqref{assume-semicomplete-split}. 
We thus conclude that
\[
|V(Q_1[a,t_2])\cap V(P_{k_0})|, |V(Q_2[b,t_2])\cap V(P_{k_0})|, |V(Q_3[s_2,c])\cap V(P_{k_0})|\ge 1.
\]
{This guarantees that we may select \(x_1 \in V(Q_1(a,t_2]) \cap V(P_{k_0})\) to minimize \(|P_{k_0}[x_1,t_1]|\), \(x_2 \in V(Q_2(b,t_2]) \cap V(P_{k_0})\) to minimize \(|P_{k_0}[x_2,t_1]|\) and \(x_3 \in V(Q_3[s_2,c)) \cap V(P_{k_0})\)  to minimize \(|P_{k_0}[s_1,x_3]|\) 

If $|P_{k_0}[s_1,x_q]| > |P_{k_0}[s_1,x_3]|$, for some $q\in [2]$, then the $(s_2,t_2)$-path $Q = Q_3[s_2, x_3] \circ P_{k_0}[x_3, x_q] \circ Q_q[x_q, t_2]$ is disjoint from $P'$, contradicting \eqref{assume-semicomplete-split}.
Thus, we obtain
\[
\max\{|P_{k_0}[s_1,x_1]|,|P_{k_0}[s_1,x_2]| < |P_{k_0}[s_1,x_3]|.
\]


\begin{claim}\label{Locally semicomplete split: claim 7}
    If $|P_{k_0}[s_1,x_1]|<|P_{k_0}[s_1,x_2]|$, then
    $x_1, x_3 \in V_1$ and $x_2 \in V_2$. Furthermore $P_{k_0}=s_1x_1x_2x_3t_1$
\end{claim}

\begin{proof}

Let $u$ denote the predecessor of $x_2$ on $P_{k_0}$. We observe that if there exists an arc $x_3 \to x_1$ or $x_3 \to u$, then we can construct an $(s_2,t_2)$-path
\[
Q' = Q_3[s_2,x_3] \circ x_3x_1 \circ Q_1[x_1,t_2]
\]
or
\[
Q'' = Q_3[s_2,x_3] \circ x_3u \circ P_{k_0}[u,x_2] \circ Q_2[x_2,t_2]
\]
that is disjoint from $P'$, which yields a contradiction to \eqref{assume-semicomplete-split}.
Therefore, we have $x_3 \nrightarrow x_1$ and $x_3 \nrightarrow u$.
Combining this with the definition of semicomplete split digraphs and the minimality of $P_{k_0}$, we conclude that $x_1, x_3, u \in V_1$ and thus $x_2 \in V_2$ (see Figure \ref{semicomplete-split-digraph-Figure3}). 

\begin{figure}[htbp]
    \centering
    \begin{subfigure}{0.455\textwidth}
        \centering
        \includegraphics[width=\linewidth]{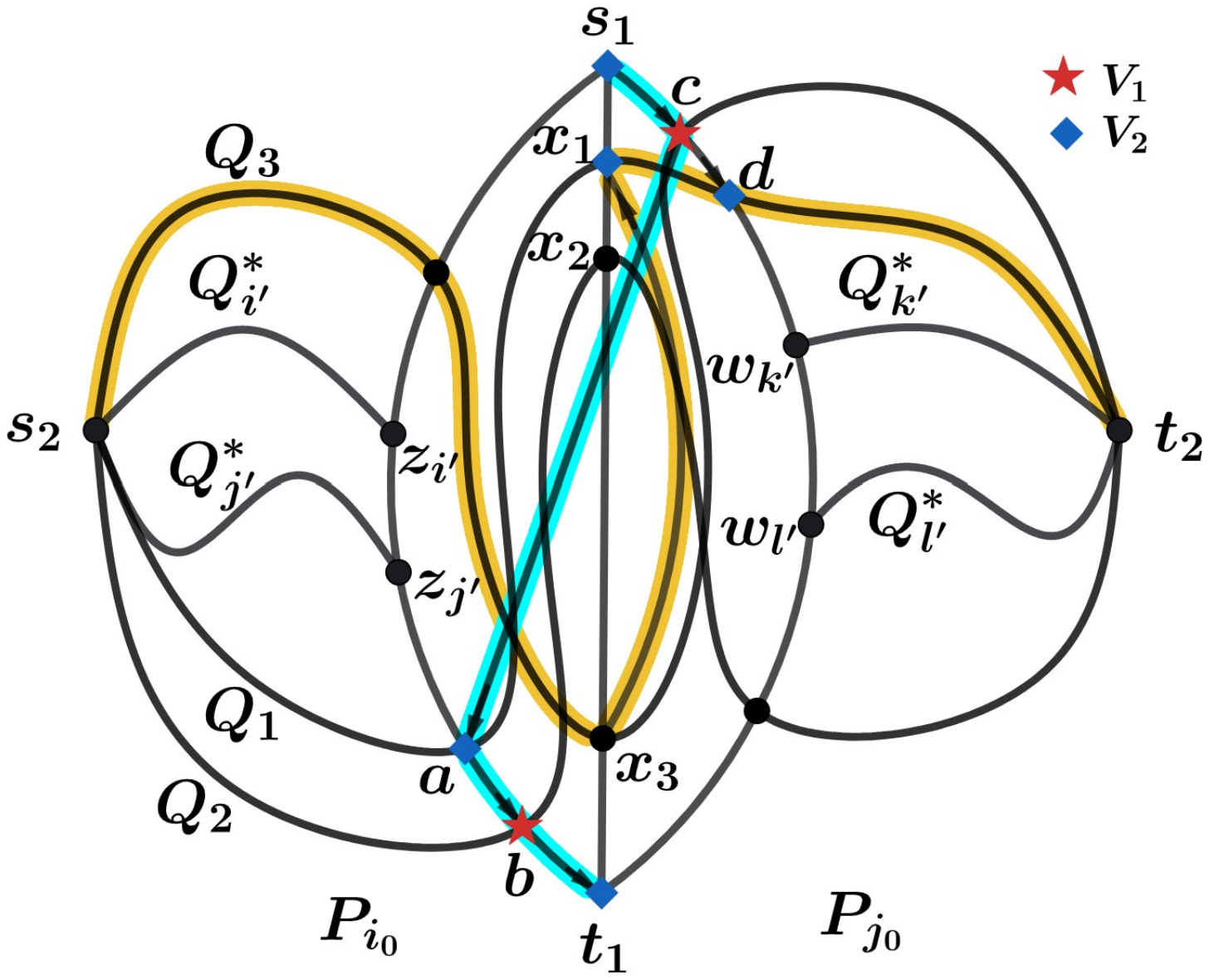} 
        \caption{}
        \label{semicomplete-split-digraph-Figure3-(a)}
    \end{subfigure}
    \hfill 
    \begin{subfigure}{0.455\textwidth}
        \centering
        \includegraphics[width=\linewidth]{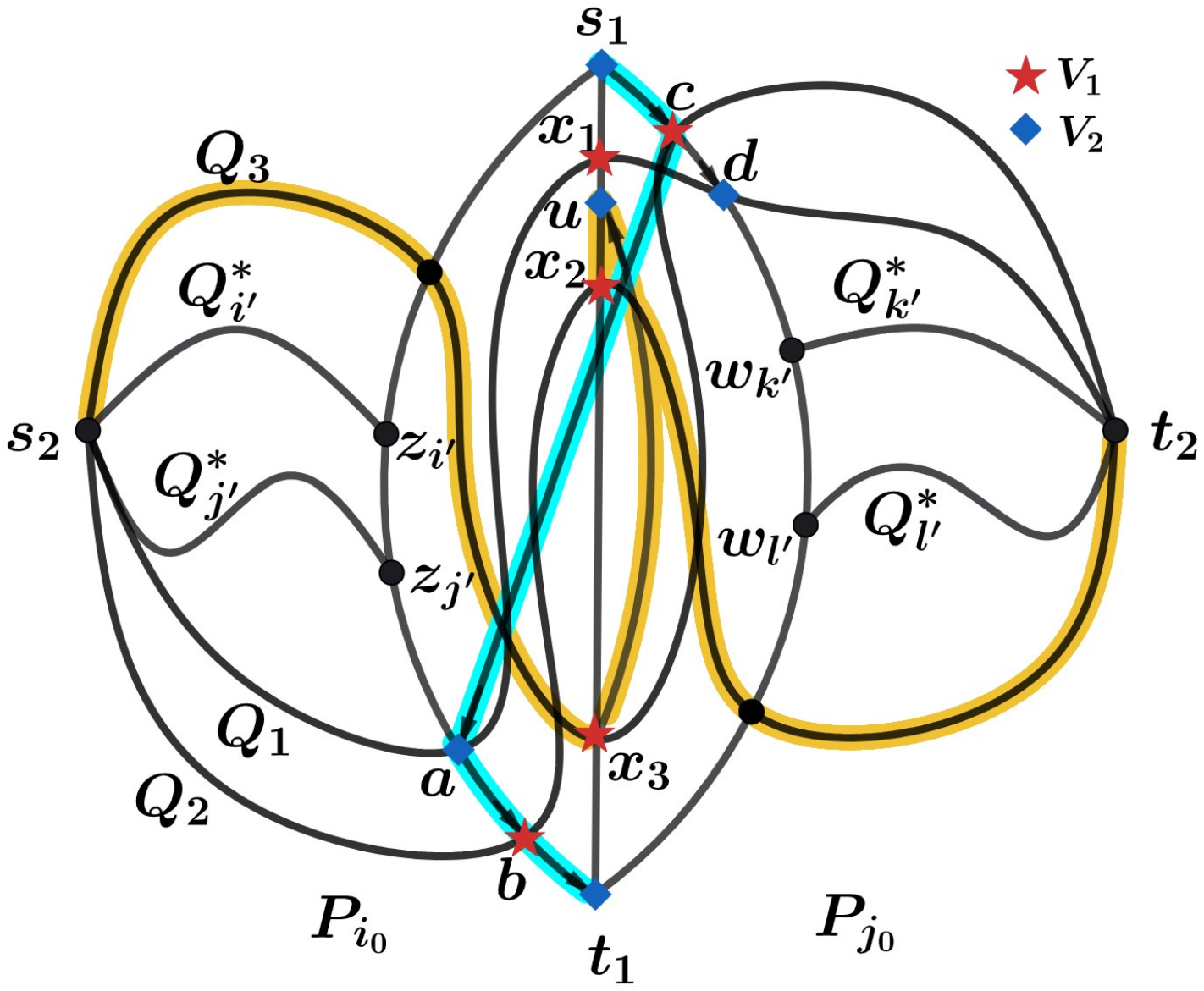}
        \caption{}
        \label{semicomplete-split-digraph-Figure3-(b)}
    \end{subfigure}
  \caption{Example with $i_0 = 1$ and $j_0 = 3$. (a) $x_3 \to x_1$; (b) $x_3 \to u$.}
    \label{semicomplete-split-digraph-Figure3}
\end{figure}

Next, we verify that $P_{k_0}[s_1,x_2] = s_1x_1x_2$. Let \(v_1 \in V_2\) denote the predecessor of \(x_1\) on \(P_{k_0}\), and let \(v_2 \in V_2\) be its successor on \(P_{k_0}\). 
We observe that if \(|P_{k_0}[s_1,x_1]| \geq 3\) or \(|P_{k_0}[x_1,x_2]| \geq 3\), then by the minimality of \(P_{k_0}\) and the definition of semicomplete split digraphs, we can construct an \((s_2,t_2)\)-path
\[
Q^* = Q_3[s_2, x_3] \circ x_3v_1 \circ v_1x_1 \circ Q_1[x_1, t_2]
\]
or
\[
Q^{**} = Q_3[s_2, x_3] \circ x_3v_2 \circ P_{k_0}[v_2,x_2] \circ Q_2[x_2, t_2]
\]
that is disjoint from \(P'\), which yields a contradiction to \eqref{assume-semicomplete-split}. 
It follows that \(P_{k_0}[s_1,x_1] = s_1x_1\) and \(P_{k_0}[x_1,x_2] = x_1x_2\). \\

If the successor $w$ of $x_3$ on $P_{k_0}$ is not $t_1$, then $Q_3[s_2,x_3]\circ P_{k_0}[x_3,w]\circ wx_1\circ Q_1[x_1,t_2]$ is disjoint from $P'$
contradiction. Hence we have $P_{k_0}[x_3t_1]=x_3t_1.$ Finally, if $|P_{k_0}[x_2,x_3]|>2$, then $x_3x_2\in A(D)$ and the path $Q_3[s_2,x_3]\circ x_3x_2\circ Q_2[x_2,t_2]$ is disjoint from $P'$, contradiction. Hence we have shown that $P_{k_0}=s_1x_1x_2x_3t_1$.
\end{proof}

\begin{claim}
    If $|P_{k_0}[s_1,x_2]|<|P_{k_0}[s_1,x_1]|$, then
    $x_2, x_3 \in V_1$ and $x_1 \in V_2$. Furthermore, $P_{k_0}=s_1x_2x_1x_3t_1$
\end{claim}

\begin{proof}
    The proof is analogous to the proof of Claim \ref{Locally semicomplete split: claim 7} by interchanging the roles of $x_1$ and $x_2$.
\end{proof}

Suppose first that $P_{k_0}=s_1x_1x_2x_3t_1$.
By Claims \ref{Locally semicomplete split:claim 2} and \ref{Locally semicomplete split: claim 7}, \(x_2\) is adjacent to \(c\) and \(x_3\) is adjacent to \(d\). If \(c \to x_2\), then the \((s_1,t_1)\)-path \(P = s_1 c x_2 x_3 t_1\) is disjoint from \(Q_1\), so we must have \(x_2 \to c\). If \(d \to x_3\), then the \((s_1,t_1)\)-path \(P = s_1 c d x_3 t_1\) is disjoint from \(Q_2\), another contradiction to \eqref{assume-semicomplete-split}. Hence,  \(x_3 \to d\). However, combining the arc relations \(x_2 \to c\) and \(x_3 \to d\), we see that the \((s_1,t_1)\)-path \(P = s_1 x_1 x_2 c a b t_1\) is disjoint from the \((s_2,t_2)\)-path \(Q = Q_3[s_2, x_3] \circ x_3 d \circ P_{j_0}[d, w_{l'}] \circ Q_{l'}^*\), which again leads to a contradiction to \eqref{assume-semicomplete-split}. Now consider the case when $P_{k_0}=s_1x_2x_1x_3t_1$. Interchanging the roles of $x_1$ and $x_2$ above we obtain that $x_1a,x_3d\in A(D)$ and then the $(s_1,t_1)$-path $s_1x_2x_1cabt$ is disjoint from the $(s_2,t_2)$-path $Q$ above, leading to the final contradiction.


\vspace{4mm}

\mycase{There exists an index \(i_0' \in \intthree\) such that \(|\{z_1,z_2,z_3\} \cap V(P_{i_0})| \geq 2\) and \(|\{w_1,w_2,w_3\} \cap V(P_i)| < 2\) for every \(i \in \intthree\).}{Case 2}

In this case, \(|\{w_1,w_2,w_3\} \cap V(P_i)| = 1\) for all \(i \in \intthree\). We thus let \(w_{i_1} \in V(P_1)\), \(w_{i_2} \in V(P_2)\), and \(w_{i_3} \in V(P_3)\) where \(\{i_1, i_2, i_3\} = \{1,2,3\}\).
Without loss of generality, assume that \(i_0' = 1\) and \(\{z_{i'}, z_{j'}\} \subseteq \{z_1,z_2,z_3\} \cap V(P_1)\), and further that \(|P_1[s_1, z_{i'}]| < |P_1[s_1, z_{j'}]|\).
The vertex $w_{i_1}$ appears on $P_1$ before each of $z_{i'}, z_{j'}$ as otherwise $Q^*_{i'}\circ P_1[z_{i'},w_{i_1}]\circ Q_{i_1}[w_{i_1},t_2]$ is disjoint from $P_2$.

Let \(b\) be the predecessor of \(t_1\) on \(P_1\) (i.e., \(b = s_1^1\)) and \(a\) be the predecessor of \(b\) on \(P_1\). Fix an arbitrary path \(P_i\) with \(i \in \{2,3\}\) distinct from \(P_1\), without loss of generality, let this path be \(P_3\). Let \(c\) be the successor of \(s_1\) on \(P_3\) (i.e., \(c = s_1^3\)) and \(d\) be the successor of \(c\) on \(P_3\). By Claim \ref{Locally semicomplete split:claim 2}, we have \(b, c \in V_1\) and \(a, d \in V_2\). We denote \(Q_{i'}^* = Q_{i'}[s_2, z_{i'}]\) and \(Q_{j'}^* = Q_{j'}[s_2, z_{j'}]\). If} \(a \to c\), then the \((s_2, t_2)\)-path \(Q = Q_{i'}^* \circ P_1[z_{i'}, a] \circ ac \circ P_3[c, w_{i_3}] \circ Q_{i_3}[w_{i_3}, t_2]\) is disjoint from \(P_2\), contradicting \eqref{assume-semicomplete-split}. We thus conclude that \(c \to a\) (See Figure \ref{semicomplete-split-digraph-Figure4-(a)}).

As in the previous case we observe that each $Q_i$, $i\in [3]$ contains exactly one of the vertices $a,b,c$.

Without loss of generality, we assume \(a\in V(Q_1)\), \(b\in V(Q_2)\), and \(c\in V(Q_3)\). If \(V(Q_3[s_2,c]) \cap V(P_1) = \emptyset\), then the \((s_2,t_2)\)-path \(Q = Q_3[s_2,c] \circ P_3[c,w_{i_3}] \circ Q_{w_{i_3}}[w_{i_3}, t_2]\) is disjoint from \(P_1\), which yields a contradiction to \eqref{assume-semicomplete-split}. Thus, \(|V(Q_3[s_2,c]) \cap V(P_1)| \geq 1\). Further, let \(u\) be a vertex such that \(u \in V(Q_3[s_2, c]) \cap V(P_1)\). Clearly, \(w_{i_1} \notin \{a, b, u\}\) and \(\{a, b, u, w_{i_1}\} \subseteq V(P_1)\). Note that both \(z_{i'}\) and \(z_{j'}\) might lie in \(\{a, b, u\}\). By Claim \ref{Locally semicomplete split:claim 2} and the minimality of \(P_1\), we have \(a \to s_1^1\) (See Figure \(\ref{semicomplete-split-digraph-Figure4-(b)}\)). Thus the \((s_2,t_2)\)-path \(Q = Q_{i'}^* \circ P_1[z_{i'},a] \circ as_1^1 \circ P_1[s_1^1,w_{i_1}] \circ Q_{i_1}[w_{i_1},t_2]\) is disjoint from every \((s_1,t_1)\)-path \(P_i\) for all \(i \in \{2,3\}\), which again contradicts \eqref {assume-semicomplete-split}. 

\begin{figure}[htbp]
    \centering
    \begin{subfigure}{0.455\textwidth}
        \centering
        \includegraphics[width=\linewidth]{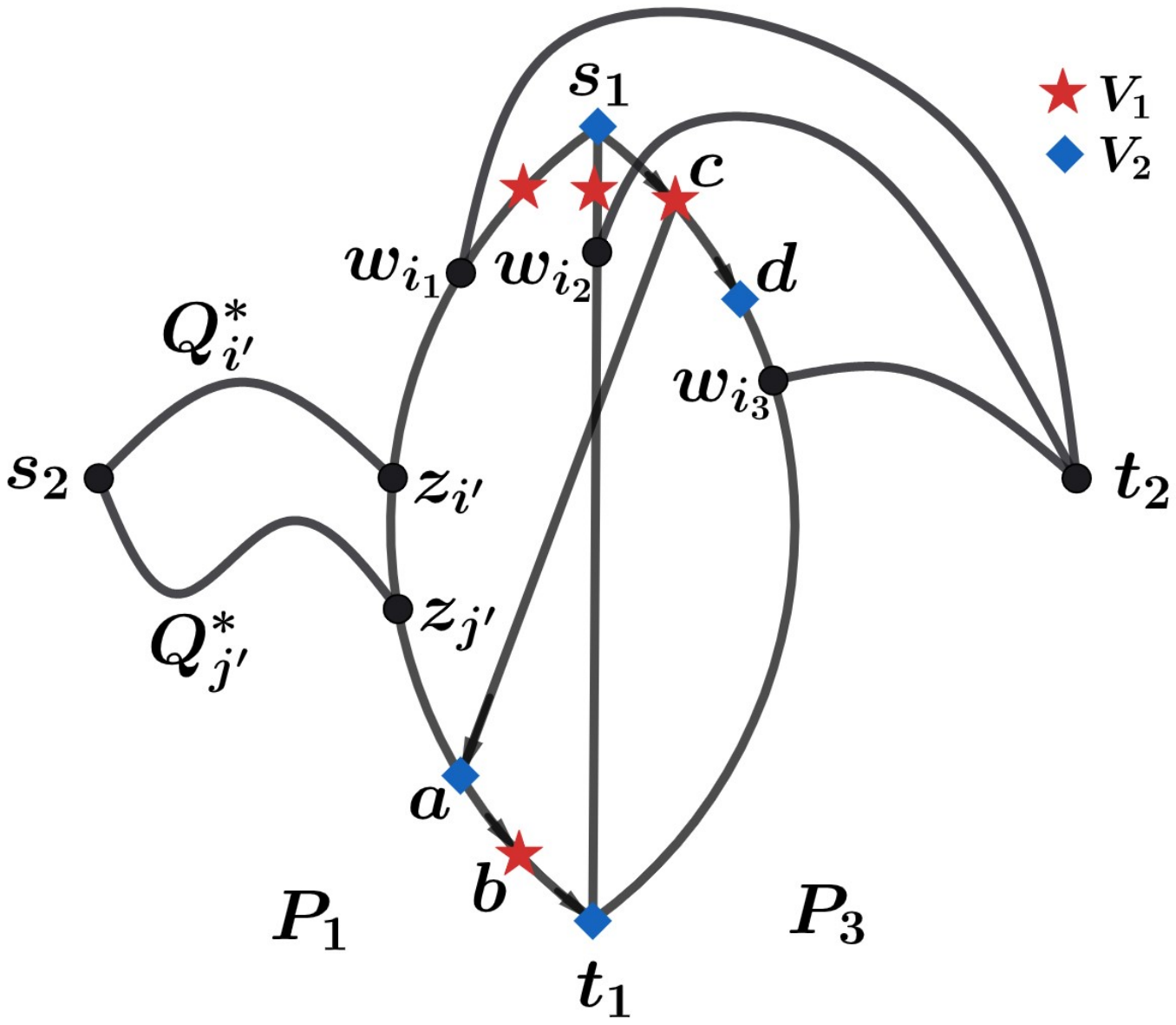} 
        \caption{}
        \label{semicomplete-split-digraph-Figure4-(a)}
    \end{subfigure}
    \hfill 
    \begin{subfigure}{0.455\textwidth}
        \centering
        \includegraphics[width=\linewidth]{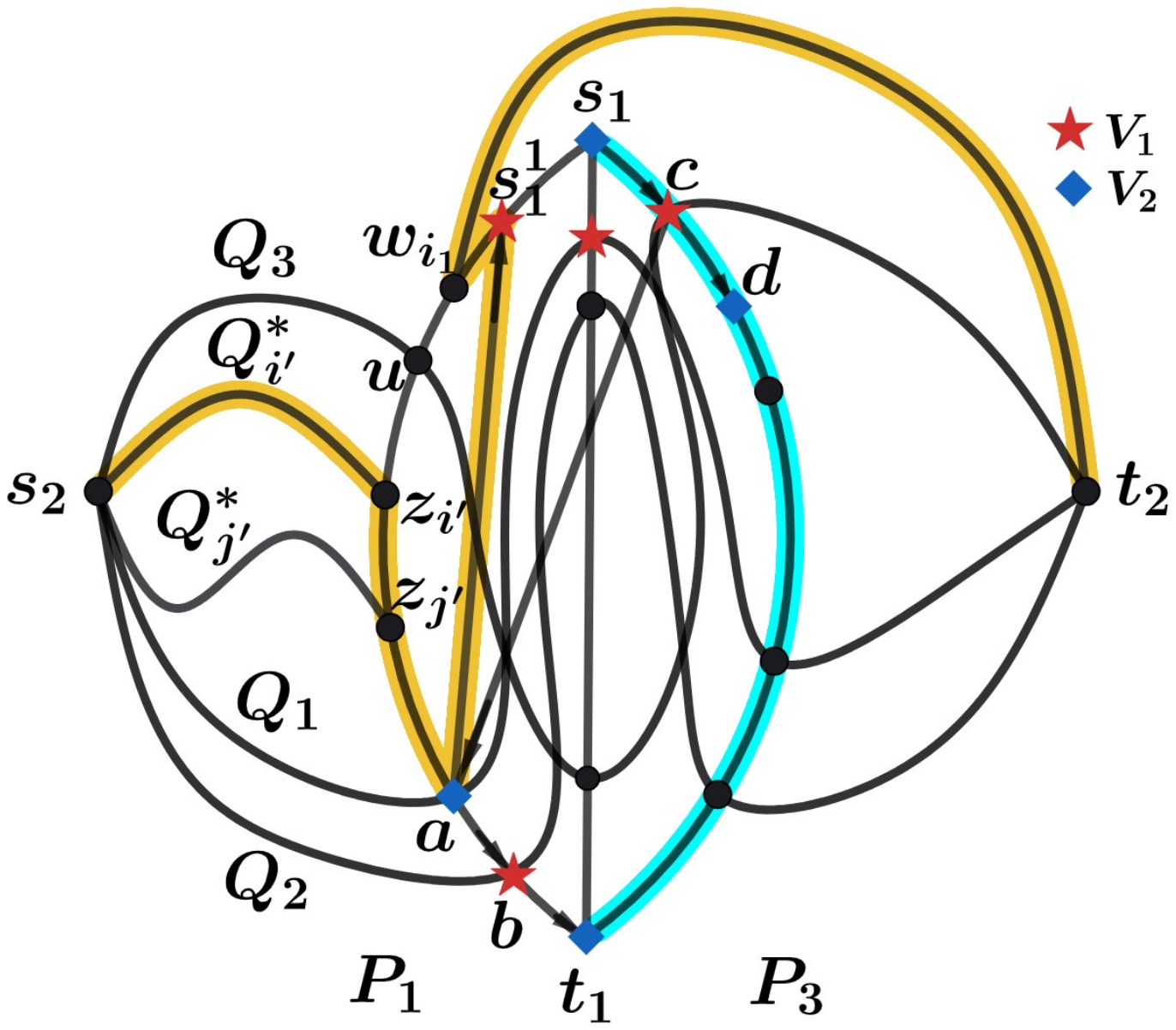}
        \caption{}
        \label{semicomplete-split-digraph-Figure4-(b)}
    \end{subfigure}
   \caption{Example with \(i_0' = 1\). (b) illustrates that \(|\{a,b,u,w_{i_1}\} \cap V(P_1)| = 4\).}
    \label{semicomplete-split-digraph-Figure4}
\end{figure}

In the case when \(|\{w_1,w_2,w_3\} \cap V(P_{i_0})| \geq 2\)   for some \(i_0 \in \intthree\) and \(|\{z_1,z_2,z_3\} \cap V(P_i)| < 2\) for every \(i \in \intthree\) we obtain the desired contradiction by reversing all arcs of \(D\) and denote the resulting digraph by \(D'\). By the proof above, \(D'\) contains a pair of disjoint \((t_1, s_1)\)-paths and \((t_2, s_2)\)-paths. This implies that \(D\) contains a pair of disjoint \((s_1, t_1)\)-paths and \((s_2, t_2)\)-paths, which contradicts \eqref{assume-semicomplete-split}.


\vspace{4mm}
\mycase{For every \(i\in\intthree\), \(|\{z_1,z_2,z_3\}\cap V(P_i)|<2\) and \(|\{w_1,w_2,w_3\}\cap V(P_i)|<2\).}{Case 3}

It follows that for each \(i\in\intthree\), \(|\{z_1,z_2,z_3\}\cap V(P_i)|=1\) and \(|\{w_1,w_2,w_3\}\cap V(P_i)|=1\). Recall that \(s_1^1\), \(s_1^2\), \(s_1^3\) are the successors of \(s_1\) on \(P_1\), \(P_2\), \(P_3\) respectively, and \(t_1^1\), \(t_1^2\), \(t_1^3\) are the predecessors of \(t_1\) on \(P_1\), \(P_2\), \(P_3\) respectively. By Claim \ref{Locally semicomplete split:claim 2}, we have \(\{s_1^1, s_1^2, s_1^3, t_1^1, t_1^2, t_1^3\} \subseteq V_1\). Using the minimality of \(P_1, P_2, P_3\) and the definition of semicomplete split digraphs, we verify that each \(P_i\) has exactly five vertices for \(i\in \intthree\), as established in Claim \ref{Locally semicomplete split: claim 10}.

\begin{claim}\label{Locally semicomplete split: claim 10}
    \(|V(P_1)| = |V(P_2)| = |V(P_3)| = 5\). Furthermore, for each \(i\in\intthree\), we have \(P_i = s_1 s_1^i y_i t_1^i t_1\). Moreover, \(y_i\in V_2\).
\end{claim}

\begin{proof}
    Suppose \(|V(P_1)| \geq 6\). Let \(u_1\) be the successor of \(s_1^1\) on \(P_1\) and \(v_1\) the predecessor of \(t_1^1\) on \(P_1\). Since \(\{s_1^1,t_1^1\} \subseteq V_1\), it follows that \(\{u_1, v_1\} \subseteq V_2\). By the minimality of \(P_1\) and the definition of semicomplete split digraphs, we have \(t_1^1u_1, v_1s_1^1 \in A(D)\). Let \(i,j \in \{1,2,3\}\) be such that \(\{z_{i}, w_{j}\} \subseteq V(P_1)\).

    If \(|P_1[s_1, z_{i}]| < |P_1[s_1, w_{j}]|\), then the \((s_2, t_2)\)-path 
    \[
    Q = Q_{i}[s_2, z_{i}] \circ P_1[z_{i}, w_{j}] \circ Q_{j}[w_{j}, t_2]
    \]
    is disjoint from \(P_r\) for all \(r \in \{2, 3\}\), which contradicts \eqref{assume-semicomplete-split}. Thus, we may assume \(|P_1[s_1, z_{i}]| > |P_1[s_1, w_{j}]|\). However, the $(s_2,t_2)$-trail
    \[
    Q = Q_{i}[s_2, z_{i}] \circ P_1[z_{i}, t_1^1] \circ t_1^1 u_1 \circ P_1[u_1, v_1] \circ v_1 s_1^1 \circ P_1[s_1^1, w_{j}] \circ Q_{j}[w_{j}, t_2]
    \]
    is disjoint from \(P_i\) for all \(i \in \{2, 3\}\), which again contradicts \eqref{assume-semicomplete-split}.


    Therefore, we conclude \(|V(P_1)| \leq 5\). By \eqref{length-semicomplete-split}, we have \(|V(P_1)| = 5\). Hence, \(P_1\) can be written as \(P_1 = s_1 s_1^1 y_1 t_1^1 t_1\). Since \(\{s_1^1,t_1^1\} \subseteq V_1\), it follows that \(y_1 \in V_2\). Analogously, the same conclusion holds for \(P_2\) and \(P_3\).
\end{proof}

From the definition \(S = \bigcup_{i=1}^{3}V(P_i) - \{s_1,t_1\}\) and Claim \ref{Locally semicomplete split: claim 10}, we have \(S = \{s_1^1,s_1^2,s_1^3,y_1,y_2,y_3,t_1^1,t_1^2,t_1^3\}\). By the symmetry between the paths \(Q_1, Q_2, Q_3\) and the paths \(P_1, P_2, P_3\), the conclusion of Claim \ref{Locally semicomplete split: claim 10} also holds for \(Q_1, Q_2, Q_3\). It follows that \(S' = \bigcup_{i=1}^{3}V(Q_i) - \{s_2, t_2\}\) satisfies \(|S'| = 9\). Combining this with \eqref{intersection-semicomplete-split} and Claim \ref{Locally semicomplete split: claim 10}, we obtain \(S' = S = \{s_1^1, s_1^2, s_1^3, y_1, y_2, y_3,t_1^1, t_1^2, t_1^3\}\). 

\begin{claim}\label{semicomplete-split digraph-x_1,x_2,x_3}
    For each \(i \in \intthree\), \(|\{y_1,y_2,y_3\} \cap V(Q_i)| = 1\).
\end{claim}

\begin{proof}
Suppose that some $Q_i$ contains at least two of the vertices $y_1,y_2,y_3$. Then there is another $Q_j$  whose three internal vertices are all in $\{s_1^1,s_1^2,s_1^3,t_1^1,t_1^2,t_1^3\}\subseteq V_1$. This is impossible as $V_1$ is an independent set. Hence every $Q_i$ contains exactly one of the vertices $y_1,y_2,y_3$.

\end{proof}

Furthermore, combining Claim \ref{semicomplete-split digraph-x_1,x_2,x_3} with the fact that Claim \ref{Locally semicomplete split: claim 10} holds for \(Q_1,Q_2,Q_3\), the paths \(Q_1, Q_2, Q_3\) must take the following forms:
\[
\begin{aligned}
Q_1 &= s_2 \to s_2^1 \to y_{j_1} \to t_2^1 \to t_2, \\
Q_2 &= s_2 \to s_2^2 \to y_{j_2} \to t_2^2 \to t_2, \\
Q_3 &= s_2 \to s_2^3 \to y_{j_3} \to t_2^3 \to t_2,
\end{aligned}
\tag{$\ast$}
\label{eq:Q_path_form}
\]
where $\{j_1,j_2,j_3\} = \{1,2,3\}$ and 
$\{s_2^1,s_2^2,s_2^3,t_2^1,t_2^2,t_2^3\}=\{s_1^1,s^1_2,s_1^3,t_1^1,t_1^2,t_1^3\}$.\\

\begin{claim}\label{cl:extra}
$\{s_2^1,s_2^2,s_2^3\}=\{t_1^1,t_1^2,t_1^3\}$  and 
$\{t_2^1,t_2^2,t_2^3\}=\{s_1^1,s^1_2,s_1^3\}$.
\end{claim}
\begin{proof}
    Suppose this not the case. Then w.l.o.g. $s_2^1=s_1^1$. Then we also have $t_1^1\in \{s_2^1,s_2^2,s_2^3\}$, since otherwise the $(s_2,t_2)$-path $s_2s_1^1y_1t_1^1t_2$ is disjoint from both $P_2$ and $P_3$. W.l.o.g. $s_2^2=t_1^1$ and now we see by the argument we just gave that w.l.o.g. $s_2^3=t_1^2$.
    The structure of $Q_1,Q_2,Q_3$ as described in (\ref{eq:Q_path_form}) now implies that $y_1 \to t_1^j$ for $j=2$ or $j=3$ and thus 
the path $s_2s_1^1y_1t_1^jt_2$ is disjoint from either $P_2$
 or $P_3$, contradiction. 
 \end{proof}

Up to isomorphism, the semicomplete split digraph \(D\) is illustrated in Figure \ref{semicomplete-split-digraph-Figure5}, where only the key arcs are displayed, the remaining arcs will be further analyzed in subsequent discussions.

\begin{figure}[!htbp] 
    \centering 
    \includegraphics[width=0.46\textwidth]{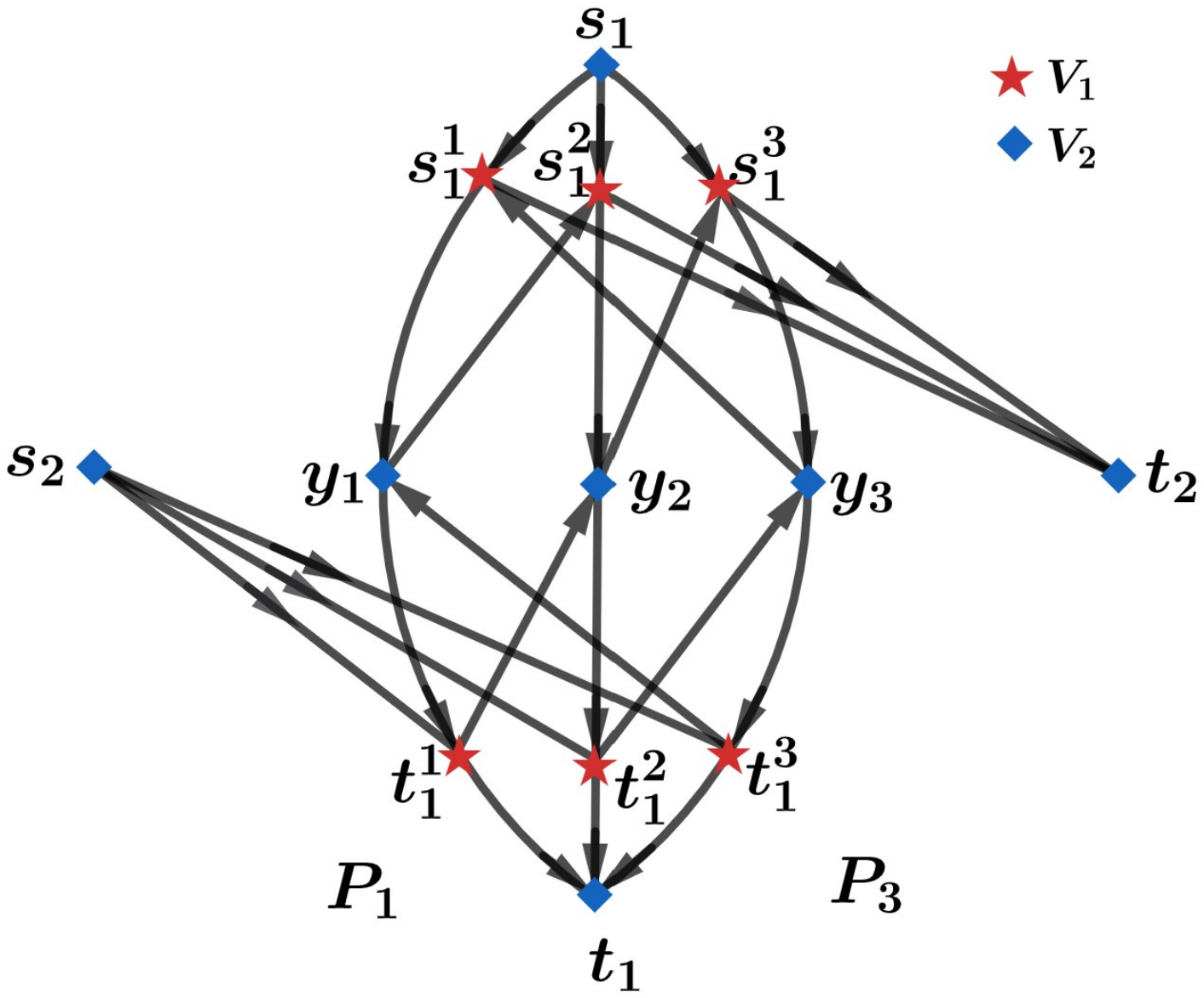}
    \caption{}
    \label{semicomplete-split-digraph-Figure5}
\end{figure}

From the definition of semicomplete split digraphs, $y_2$ is adjacent to $s_1^1$.
However, both orientations of the arc between them, lead to a contradiction with (\ref{assume-semicomplete-split}).
Indeed, if $y_2 \to s_1^1$, then the $(s_2,t_2)$-path $Q = s_2 t_1^1 y_2 s_1^1 t_2$ is disjoint from $P_3$.
If $s_1^1 \to y_2$, then the $(s_2,t_2)$-path $Q' = s_2 t_1^3 y_1 s_1^2 t_2$ and the $(s_1,t_1)$-path $P = s_1 s_1^1 y_2 t_1^2 t_1$ are disjoint.
In either case, we obtain a contradiction and the proof is complete.

\qed

\section{Proof of Theorem \ref{theorem:6-strong semicomplete multipartite}}\label{section 5}
   \noindent \textbf{Proof of Theorem \ref{theorem:6-strong semicomplete multipartite}.}
Let \(D\) be a \(6\)-strong semicomplete multipartite digraph with partite sets \(V_1, V_2, \dots, V_c\). {If there exists an $(s_1,t_1)$-path $P$ of length at most 4 in $D - \{s_{2},t_{2}\}$, then $D - P$ is strong. It follows that \(D - P\) contains an \((s_{2}, t_{2})\)-path. Hence, there are two disjoint paths as desired. So, from now on, we assume that 
\begin{equation}\label{equation5}
    \text{there is no $(s_1,t_1)$-path of length at most 4 in $D - \{s_{2},t_{2}\}$.} 
\end{equation}}

Let \( P_1, P_2, P_3, P_4 \) be  four internally disjoint \( (s_1, t_1) \)-paths in \( D - \{s_2, t_2\} \) each  selected to be minimal. Similarly, let $Q_1$ be a minimal $(s_2, t_2)$-path in $D - \{s_1, t_1\}$. We can assume that  $Q_1$  intersects all $P_1, P_2, P_3, P_4$; otherwise, we are done. Let $S=\bigcup_{i=1}^4 V(P_i) \setminus \{s_1,t_1\}$ and {let $z$ (resp.\ $w$) denote the first (resp.\ last) vertex in $V(Q_1) \cap S$ along the path $Q$}. Without loss of generality, we assume that the vertex \( z\in V( P_1) \), and \( w \in V( P_{j_0}) \) for some \( j_0 \in [4] \). Let $t_1^-$ be the predecessor of $t_1$ on $P_1$, and let $s_1^+$ (resp. $s_1^{++})$ be the successor of $s_1$ (resp. $s_1^+)$ on $P_{j_0}$. {Suppose that  \(t_1^- \in V_i\) and \(s_1^+ \in V_j\) for some \(i,j \in [c]\). If  \(i \neq j\), let $s':=s^+_1$, otherwise, let $s':=s_1^{++}$. This implies that $s' \notin V_i$, so $t_1^-$ is adjacent to $s'$. Together with (\ref{equation5}), we thus have $t_1^- s' \in A(D)$.}



If \(w \neq s_1^+\), then the $(s_2,t_2)$-path \(Q_1[s_2,z] \circ P_1[z,t_1^-] \circ t_1^-s_1^{++} \circ P_{j_0}[s_1^{++},w] \circ Q_1[w,t_2]\) is disjoint from \(P_j\) for some \(j \in \intfour \setminus \{1, j_0\}\). In the remaining case we have  \(w = s_1^+\). Let \(t_1'\) {be} the predecessor of \(t_1\) on \(P_{j_0}\). By the minimality of \(P_{j_0}\) and (\ref{equation5}), there exists a vertex \(u\) on \(P_{j_0}[s_1^{++}, t_1']\) such that \(us_1^+ \in A(D)\). Thus, the $(s_2,t_2)$-path \(Q_1[s_2,z] \circ P_1[z,t_1^-] \circ t_1^-s_1^{++} \circ P_{j_0}[s_1^{++},u] \circ us_1^+ \circ Q_1[s_1^+,t_2]\) is disjoint from \(P_j\) for some \(j \in \intfour \setminus \{1, j_0\}\). {In  conclusion}, we can obtain a pair of disjoint \((s_1,t_1)\)-path and \((s_2,t_2)\)-path in \(D\), as desired.
\qed

\section{Remarks}\label{sec:remarks}

The 2‐linkage problem can be solved in polynomial time for semicomplete digraphs \cite{bang1992polynomial}. Inspired by this, Bang-Jensen and Wang \cite{bang2025strong} posed the question of whether the 2-linkage problem for split digraphs is also polynomial-time solvable.

\begin{problem}\cite{bang2025strong}\label{open-problem 1}
Is there a polynomial algorithm for the 2-linkage problem for split digraphs?
\end{problem}

Note that the answer is yes for semicomplete split digraphs. This follows from a result by Chudnovsky, Scott and Seymour \cite{CHUDNOVSKY2015582} combined with the fact that semicomplete split digraphs are a subclass of semicomplete multipartite digraphs (for which polynomial algorithms for the \(2\)-linkage problem are already known).

Theorem \ref{split digraph:main theorem 1} says that every 6-strong split digraph is 2-linked. While the tightness of the bound 6 remains open, we believe that 6 is indeed tight. We therefore formulate the following problem.

\begin{problem}
Does there exist a 5-strong split digraph that fails to be 2-linked?
\end{problem}

{Recall that Zhou and Yan \cite{zhou2025proof} proved that every $(2k+1)$-strong semicomplete digraph with minimum out-degree at least $7k^2+36k$ is $k$-linked. We believe that an analogous result holds for split digraphs.}

\begin{problem}\label{probelm-split:(2k+1) and ck^2}
   {There is a constant \(c\) such that every \((2k+1)\)-strong split digraph \(D\) with \(\delta^+ (D) \geq ck^2\) is \(k\)-linked.}
\end{problem}

\section*{Acknowledgments}
This work was supported by the following grants: National Natural Science Foundation of China (12571373), project ZR2025MS05 supported by Shandong Provincial Natural Science Foundation.

\bibliography{references} 

\end{document}